\documentclass[a4paper, 10pt]{amsproc}
\usepackage{defs}
\usepackage{orcidlink}
\usepackage{calc}
\usepackage{graphicx,wrapfig,lipsum}
\usepackage[dvipsnames]{xcolor}
\usepackage{tabularray}
\usepackage{empheq} 
\newtcolorbox{empheqboxed}{colback=white, 
    colframe=black,
    boxrule=0.25mm,
    width=\columnwidth,
    sharpish corners,
    top=-2mm, 
    left=2pt,
    bottom=5pt
}
\definecolor{metablue}{HTML}{0064E0}
\definecolor{metafg}{HTML}{1C2B33}
\definecolor{metabg}{HTML}{F1F4F7}
\definecolor{metabgdeep}{HTML}{D9EFFF}
\definecolor{metagreen}{HTML}{EAFFE8}
\definecolor{metagreen}{HTML}{FCFFEE}
\definecolor{metared}{HTML}{FFEAE8}

\newtheorem{theorem}{Theorem}

\newtheorem{remark}{Remark}

\newtheorem{proposition}{Proposition}
\newtheorem{lemma}{Lemma}
\newtheorem{corollary}{Corollary}

\DeclareSymbolFont{extraup}{U}{zavm}{m}{n}
\DeclareMathSymbol{\varheart}{\mathalpha}{extraup}{86}
\DeclareMathSymbol{\vardiamond}{\mathalpha}{extraup}{87}
\DeclareMathSymbol{\varclub}{\mathalpha}{extraup}{84}
\DeclareMathSymbol{\vardspade}{\mathalpha}{extraup}{85}

\usepackage[framemethod=tikz,xcolor=true]{mdframed}

\newmdenv[backgroundcolor=metabgdeep, roundcorner=10pt, skipabove=4pt, linewidth=0pt, innertopmargin=4pt]{myframe}

\newmdenv[backgroundcolor=metabgdeep, roundcorner=10pt, skipabove=4pt, linewidth=1pt, innertopmargin=4pt]{myOCP}

\newmdenv[backgroundcolor=metared, roundcorner=10pt, skipabove=7pt, linewidth=0pt, innertopmargin=7pt]{myalgo}
\newmdenv[%
    leftmargin=0.5cm,
    backgroundcolor=yellow!10,%
    roundcorner=5pt,%
    tikzsetting={draw=red, line width=2.0pt}%
    ]{SpecialText}%

    \newmdenv[
  backgroundcolor=metabg,
  roundcorner=5pt,
  skipabove=4pt,
  linewidth=0.5pt,
  innertopmargin=3pt,
  tikzsetting={draw=black, line width=0.5pt},
  frametitlefont=\bfseries,
  frametitlebackgroundcolor=blue!10,
  frametitleaboveskip=0.5ex,
  frametitlebelowskip=0.5ex
]{myOCPboxTitle}

\usepackage{cleveref}

\newcounter{ocpproblem}

\crefname{ocpproblem}{Problem}{Problems}
\Crefname{ocpproblem}{Problem}{Problems}

\newenvironment{problembox}[1][]{%
  \refstepcounter{ocpproblem}%
  \begin{myOCPboxTitle}[frametitle={Problem~\theocpproblem#1}]%
}{%
  \end{myOCPboxTitle}%
}

\title[UnBalOT-DenCon]{Globally Solving Unbalanced Optimal Transport and Density Control for Gaussian Distributions}

\author[H. Nakashima]{Haruto Nakashima}
\author[S. Ganguly]{Siddhartha Ganguly\,\orcidlink{0000-0003-2046-2061}}
\author[K. Kashima]{Kenji Kashima\,\orcidlink{0000-0002-2963-2584}}

\thanks{%
	S. Ganguly is with \faGroup\ Daniel Guggenheim School of Aerospace, \faUniversity\ Georgia Institute of Technology, \faMapMarker\  Atlanta, USA. N. Nakashima and K. Kashima are with \faGroup\ The Applied Mathematics and Physics Department, Graduate School of Informatics, \faUniversity\ Kyoto University, \faMapMarker\  Kyoto, Japan.}

\thanks{%
	Contact Information: (HN): \faEnvelope\ \texttt{nakashima.haruto.72v@st.kyoto-u.ac.jp}, (SG) \faHome\ \url{https://sites.google.com/view/siddhartha-ganguly}, \faEnvelope\ \texttt{sganguly41@gatech.edu}, (KK): \faEnvelope\ \texttt{kk@i.kyoto-u.ac.jp}, \faHome\ \url{https://www.bode.amp.i.kyoto-u.ac.jp/en/kk/}.
}

\begin{document}

\maketitle

\begin{abstract}
In this article, we study unbalanced optimal transport (UOT) and establish a control-theoretic dynamical extension, which we call \emph{the unbalanced density control} (UDC), for a class of Gaussian reference measures. In the static setting, we consider UOT with quadratic transport cost and Kullback--Leibler penalties on the marginals relative to prescribed Gaussian measures. We show that the infinite-dimensional variational problem admits an \emph{exact Gaussian reduction}, yielding a finite-dimensional optimization over masses, means, and covariances, together with a closed-form expression for the optimal transported mass. We then formulate UDC for discrete-time linear systems, where the initial and terminal state measures are imposed softly through KL penalties and the intermediate evolution is governed by controlled linear dynamics with quadratic control cost. For this problem, we prove that any feasible solution can be replaced, without loss of optimality, by a Gaussian initial measure and an affine-Gaussian control policy. This leads to an exact finite-dimensional reformulation and, after a standard covariance-steering lifting, to an SDP-based optimization for fixed mass, again coupled with a closed-form mass update. We further establish existence of optimal solutions and identify a sufficient condition under which the affine-Gaussian UDC policy is deterministic. These results provide globally optimal solution methods for both Gaussian UOT and Gaussian UDC. Finally, we illustrate our results with several numerical examples. 
\end{abstract}

\section{Introduction}\label{sec:introduction}

Optimal transport (OT) provides a fundamental framework for comparing two distributions by assigning a minimal cost to transporting mass from one to the other. Originating in the work of Monge \cite{Monge1781}, subsequently relaxed by Kantorovich \cite{ref:Kantoro:prob}, and later given a dynamic fluid-mechanical interpretation by J.-D. Benamou and Y. Brenier  \cite{Benamou2000}, OT has become a central tool across mathematics \cite{villani_ot}, control theory \cite{ref:YC:TTG:MV:Part-I,ef:YC:TTG:MV:Part-II,ef:YC:TTG:MV:Part-III,ref:SOC:SB:YC,ref:Halder:Wendel:WassDensity}, image processing \cite{Papadakis2015,Gazdieva2025,Bon2025} and computer science \cite{Peyre2020,Peyre2025}. At a \emph{very informal level}, given two probability measures \(\mu\) and \(\nu\), the OT problem seeks an optimal mechanism to rearrange the mass distributed according to \(\mu\) into that prescribed by \(\nu\), while minimizing a specified transportation cost. In the classical Monge formulation \cite{Monge1781}, for the quadratic cost on \(\R^d\), one seeks a measurable map \(T:\R^d \longrightarrow \R^d\) such that \(T_{\#}\mu = \nu\), and which minimizes the transportation cost
\begin{equation}
\label{eq:MongeProblem}
\inf_{\aset[]{T \suchthat T_{\#}\mu=\nu} } \int_{\R^d} \norm{x-T(x)}^2 \,\rmd \mu(x). \tag{\textsf{(MP)}}
\end{equation}
However, an admissible transport map need not exist in general, and even when it does, the Monge formulation is often too rigid for analysis.\footnote{\(T_\#\mu\) denotes the pushforward of \(\mu\) by \(T\), which is defined by \((T_\#\mu)(A)\Let \mu(T^{-1}(A))\) for all Borel subsets \(A\subset \R^d\). } To overcome this difficulty, Kantorovich introduced a relaxed formulation in which one optimizes over transport plans, rather than deterministic maps \cite{ref:Kantoro:prob}. Specifically, for every Borel subset \(E \subset \R^d\), one considers the set of \emph{transport plans}
\begin{align}
    \Pi(\mu,\nu) \Let \left\{ 
    \pi\, \middle\vert\,
    \begin{array}{@{}l@{}}
    \mu(E) = \pi(E \times \Rd),\,
    \nu(E) = \pi(\Rd \times E)\,\,\text{for all } E\in\mathcal{B}(\Rd)\end{array}
    \right\}\nn
\end{align}
and solves the variational problem 
\begin{equation}\label{eq:KantoProblem}
\inf_{\pi \in \Pi(\mu,\nu)} \int_{\Rd}\int_{\Rd}
\norm{x-y}^2 \rmd\pi(x,y). \tag{\textsf{(KP)}}
\end{equation}
Unlike the Monge problem \eqref{eq:MongeProblem}, the formulation \eqref{eq:KantoProblem} permits mass located at a point \(x\) to be split and transported to multiple destinations, thereby yielding a substantially more flexible and more importantly, \emph{numerically viable framework}. While static formulations \eqref{eq:MongeProblem} and \eqref{eq:KantoProblem} captures the geometry of mass redistribution, dynamic formulations make the temporal aspect explicit. In particular, the same transport problem can be interpreted as the evolution of a time-dependent density (via a \emph{continuity equation}) that connects \(\mu\) to \(\nu\) over a finite time horizon \cite{Benamou2000}. This viewpoint (the so-called Benamou-Brenier formulation) established a control-theoretic formulation of OT \cite{ref:SOC:SB:YC,ref:YC:TG:MP:SB-and-OT}.
\subsection{Motivations}\label{subsec:motivation}
The classical OT problems \eqref{eq:MongeProblem} and \eqref{eq:KantoProblem} are balanced in the sense that the transport plan has prescribed source and target marginals, and hence the two measures must carry the \emph{same total mass}. In many applications, however, the two objects being compared \emph{do not have the same total mass}. This occurs, for example, when mass is created, destroyed \cite{ref:Chen:UOT:1}, dissipated \cite{ref:Chen:UOT:3}, unobserved, or only partially available. In a control-theoretic context, this issue appears naturally when the endpoint distributions are not exact boundary conditions, but desired or estimated reference profiles. For example, consider a collection of mobile agents, particles, or resources whose aggregate state is described by a finite nonnegative measure. At the initial time, the available population may be known only approximately from measurements, and its effective mass may represent the number of active agents, the amount of material, the total probability weight assigned to reliable observations, or the confidence in the initial estimate. At the terminal time, the desired profile may prescribe a target spatial distribution or coverage pattern, but not necessarily with the same total mass as the initial profile. Some agents may be unavailable, some mass may be irrelevant or unobserved, and the desired terminal profile may only be partially attainable under the available control authority.

In such a situation, a balanced density-steering or covariance-steering formulation \cite{ref:CovSteer:I,ref:CovSteer:II,ref:CovSteer:III,ref:CovSteer:V:Bakolas,ref:CovSteer:VII:Okamoto,liu2024reachability,ref:CovSteer:X:DD:CDC:NH:SG:KK} can be too restrictive. If the endpoint distributions are imposed as hard constraints, then the controller is forced to match the prescribed initial and terminal distributions exactly, even when these distributions should only be interpreted as reference profiles. Moreover, after normalizing the measures into probability distributions, one loses the information encoded in their total masses. Thus, a formulation based only on balanced probability measures cannot distinguish between a high-confidence reference profile with large mass and a weak or partial reference profile with small mass. Thus, we ask: 
\begin{quote}
   \emph{Can we establish a control formulation that transports a dynamically evolving state measure while allowing mismatch from the endpoint references in a quantified way?}  
\end{quote}
This brings us to \emph{unbalanced optimal transport}. Instead of requiring exact marginal constraints, UOT penalizes deviations from prescribed reference measures, often through Kullback--Leibler (KL) divergence terms. As a consequence, UOT furnishes a mathematically natural framework for comparing finite nonnegative measures with unequal masses. Again, at an \emph{informal level}, given two finite nonnegative measures \(\alpha,\beta\) on \(\R^d\), one seeks a transport plan (with suitable regularity) \(\pi\) on \(\R^d \times \R^d\) whose marginals \(\pi_1\) and \(\pi_2\) are not required to coincide exactly with \(\alpha\) and \(\beta\), but are instead penalized for deviating from them. A standard (informal) UOT formulation is \cite{Sejourne2023,Chapel2021,Gazdieva2024,Pham2020,CHIZAT20183090}
\begin{equation}
        \begin{aligned}
            \inf_{\pi} 
            \int_{\Rd}\int_{\Rd} & \|x_2-x_1\|^2 \,\rmd \pi(x_1,x_2)
            + \gamma \KL{\pi_1}{\alpha}
            + \gamma \KL{\pi_2}{\beta}, \nn
        \end{aligned}
    \end{equation}
where \(\gamma>0\). Here, \(\alpha\) plays the role of the source reference measure and \(\beta\) that of the target reference measure, while the KL penalties relax the hard marginal constraints of the Kantorovich problem and allow for mass variation. Applications of UOT have emerged in several fields, including ML \cite{Sejourne2023} and control \cite{ref:Chen:UOT:1,ref:Chen:UOT:2,ref:Chen:UOT:3}.

The purpose of this paper is to develop the analogous idea for controlled dynamical systems. We consider a finite nonnegative state measure evolving under a discrete-time linear system. The total mass of the controlled state measure is preserved along the dynamics, but it is not prescribed a priori by either endpoint reference. Instead, the optimization chooses the transported mass and the control policy jointly: the running cost penalizes control effort, while KL terms penalize deviation of the initial and terminal state measures from the prescribed Gaussian references. We call this problem \emph{unbalanced density control} (UDC). Thus, UDC is a density-control problem in which the controller balances three competing objectives: remaining close to the initial reference, reaching close to the terminal reference, and spending minimal control energy. This viewpoint is natural for Gaussian references. In linear systems, Gaussian distributions and their first two moments are central to covariance steering and stochastic control \cite{ref:KJ:PT:Com:Efi:CDC,ref:KO:MG:PT:Opt:CovCont:LCSS,ref:FW:SG:PT}. However, the unbalanced setting introduces an additional degree of freedom: the transported mass itself becomes a decision variable. The resulting problem therefore asks not only how the mean and covariance should be steered, but also how much mass should be transported when exact agreement with the endpoint references is neither required nor desirable.

\subsection{Natural questions}\label{subsec:intro:question:motiv} 

The preceding example motivates UDC as a soft-endpoint analogue of density steering. However, it also reveals a basic mathematical difficulty. Once the initial and terminal measures are imposed through KL penalties rather than hard constraints, the optimization is no longer a standard covariance-steering problem with prescribed endpoint moments. In the static case, UOT is posed over all finite nonnegative transport plans. In the dynamical case, UDC is posed over the initial state measure and over feedback policies. Thus, even under linear dynamics and quadratic control cost, the problem is infinite-dimensional.

Also, it is not immediate that Gaussian reference measures should lead to Gaussian optimizers in either the static or dynamical unbalanced problems. In the static UOT problem, the optimizer could in principle be a non-Gaussian transport plan. In UDC, the optimal initial measure and the optimal control policy could also be non-Gaussian, nonlinear, or genuinely randomized. This raises several natural questions:
\begin{myOCP}
\vspace{1mm}
\begin{enumerate}[label={\textup{(\(\mathsf{Q}\)-\alph*)}}, leftmargin=*, widest=b, align=left]
    \item \label{ques:1} Can one formulate a control-theoretic analogue of UOT in which a finite nonnegative state measure evolves under linear dynamics, endpoint matching is imposed softly through KL penalties, and control effort is penalized along the trajectory?
    \item \label{ques:2} If the reference measures are Gaussian, i.e., \(\alpha \Let c_{\alpha}\,\mathcal{N}(m_{\alpha},\Sigma_{\alpha})\) and \(\beta \Let c_{\beta}\,\mathcal{N}(m_{\beta},\Sigma_{\beta})\), can the static UOT problem be reduced, without loss of optimality, to a finite-dimensional optimization over mass and second-order moments?
    \item \label{ques:3} In the dynamical setting, can one similarly restrict attention, without loss of optimality, to Gaussian initial measures and affine-Gaussian control laws?
    \item \label{ques:4} If such reductions are possible, can the resulting optimization problems be solved \emph{globally}, rather than merely approximately or heuristically?
    \item \label{ques:5} In the UDC problem, does one genuinely need a randomized control covariance term in the optimal policy, or does deterministic affine feedback already suffice?
\end{enumerate}
\end{myOCP}
These questions are precisely the ones addressed in this paper. We answer \ref{ques:1} positively by proposing a concrete UDC formulation for discrete-time linear systems with soft initial and terminal KL penalties. We answer \ref{ques:2} and \ref{ques:3} by proving exact Gaussian-reduction results for the static and dynamical problems, respectively. We answer \ref{ques:4} by deriving globally solvable finite-dimensional reformulations: a convex inner problem plus a scalar mass update for UOT, and an SDP-based reduced problem plus an analogous mass update for UDC. Finally, \ref{ques:5} is answered \emph{negatively} under an additional regularity condition: if the optimal state covariance sequence is positive definite, then, under the standing nonsingularity assumption on the system matrix, the optimal affine-Gaussian UDC policy is deterministic.

\subsection{Contributions}\label{subsec:intro:contrib}

\begin{enumerate}[leftmargin=*, widest=b, align=left]
\item \textbf{Unbalanced density-control formulation.} We formulate UDC as a control-theoretic dynamical counterpart of UOT for discrete-time linear systems. The state measure evolves according to linear dynamics, the running cost is quadratic in the control, and the initial and terminal specifications are imposed through KL penalties relative to Gaussian reference measures. This provides an unbalanced analogue of density and covariance steering.

\item \textbf{Gaussian reduction and global solution for UOT.} We show that, although UOT is originally posed over general transport plans, one may restrict attention, without loss of optimality, to \emph{Gaussian marginals} and the corresponding \emph{affine Gaussian coupling}. This yields a finite-dimensional reduction in terms of mass, means, and covariances. The reduced problem separates into a convex moment optimization and a scalar mass optimization, the latter admitting a closed-form solution.

\item \textbf{Gaussian reduction and global solution for UDC.} For the dynamical problem, we prove that any feasible solution can be replaced by a Gaussian initial measure and an affine-Gaussian feedback law constructed from the first and second moments of the original solution, with no increase in cost. This leads to a finite-dimensional SDP reformulation for fixed mass, combined with a closed-form mass update.

\item \textbf{Existence and structure of optimal solutions.} We establish existence of minimizers for the reduced UOT and UDC problems. We further give a structural sufficient condition for deterministic optimality in UDC. Thus, randomized affine-Gaussian controls are useful for the reduction argument, while deterministic affine feedback suffices whenever the optimal covariance trajectory remains positive definite.

\item \textbf{Entropy-regularized extensions.} We also discuss entropic and maximum-entropy variants, including an entropic UOT model and a MaxEnt UDC problem. The same Gaussian-reduction viewpoint extends naturally to these settings and yields tractable finite-dimensional formulations.
\end{enumerate}
\vspace{1mm}
\textbf{Organization:} This article unfolds as follows. In \S\ref{sec:prob_form}, we introduce the UOT and UDC problems. In \S\ref{sec:main_result_uot}, we develop the main Gaussian-reduction framework for the UOT problem together with an optimization-based solution method; we also present several results for the entropy-regularized variant of UOT. The corresponding theoretical and algorithmic developments for the UDC problem are given in \S\ref{sec:main_result_udc}. In \S\ref{sec:maxent_udc}, we discuss the maximum-entropy variant of the UDC problem. Finally, in \S\ref{sec:NumExp}, we present several numerical examples.
\section{Problem formulation}\label{sec:prob_form}

Before formulating the primary problem, we introduce the notation that will be used throughout this article. 


\subsubsection{Notation}\label{subsec:notation}
We employ standard notations throughout the article. We write \(\R\) for the set of real numbers and \(\bbN \Let \aset[]{1,2,\dots}\) for the set natural numbers. Let \(p \ge 1\) and fix \(d \in \bbN\); over the \(d\)-dimensional Euclidean space \(\Rd\), we denote the set of finite nonnegative measures with finite \(p\)-th moments by \(\calM_p^{+}(\Rd)\). For \(\mu \in \calM_p^{+}(\Rd)\), we denote its total mass by $\int \rmd \mu$. If \(\int\rmd\mu>0\), we define the \emph{normalized} probability measure \(\bar{\mu}\Let\frac{\mu}{\int_{\Rd}\rmd \mu}\). Therefore, \(\bar \mu\) is always a probability measure. The space of integrable and nonnegative functions over \(\Rd\) will be denoted by \(L^{1}_{+}(\Rd)\). For \(m \in \Rd\) and \(\Sigma \in \mathbb{R}^{d \times d}\) with \(\Sigma \succeq \textup{O}\) (the null matrix), we denote by $\calN(m,\Sigma)$ the Gaussian probability measure with mean $m$ and covariance $\Sigma$. When \(\Sigma \succ \textup{O}\) its Lebesgue density is written as \(\calN(x;m,\Sigma)\).  For a nonnegative scalar $c\ge 0$,  we call a measure of the form \(c\,\calN(m,\Sigma)\) a (unbalanced) \emph{Gaussian measure}. For $\alpha, \beta \in  \calM_p^{+}(\Rd)$ with \(\alpha\) absolutely continuous with respect to \(\beta\), the Kullback--Leibler (KL) divergence is defined by:
\begin{equation*}\label{eq:KL_def}
    \KL{\alpha}{\beta} \coloneqq \int_{\mathbb{R}^d} \log \left(\frac{\rmd\alpha}{\rmd\beta}(x)\right) \rmd\alpha(x) - \int_{\Rd} \rmd\beta + \int_{\Rd} \rmd\alpha,
\end{equation*}
and \(\KL{\alpha}{\beta} = +\infty\) otherwise.
For a probability density \(p\) on \(\Rd\), the entropy is defined by \(\ent{p}\Let-\int_{\Rd}p(x)\log p(x) \rmd x.\) For a matrix \(X\in\R^{d\times d}\), \(X^\dagger\) denotes its Moore-Penrose inverse, and \(I\) denotes identity matrix of appropriate dimension.

\subsubsection{Reference measures} We begin this section with the standard UOT problem between two Gaussian reference measures; see \eqref{eq:reference-measures} ahead. We then introduce a control-theoretic dynamical extension of this problem, in which a distribution evolves under discrete-time linear dynamics and the mismatch with the prescribed endpoint measures is penalized through the standard Kullback--Leibler divergence. This leads to the \emph{unbalanced density control formulation} studied in this article.\footnote{Here, by ``dynamical extension'' we do not mean the standard PDE-constrained dynamical formulation of UOT; see, e.g., \cite[\S 3.2]{Sejourne2023}, \cite{CHIZAT20183090}. Rather, we mean a control-theoretic formulation in which the state distribution evolves under a prescribed linear dynamical system.}

Fix \(d \in \bbN\), and let \(c_{\alpha},c_{\beta}>0\), \(m_{\alpha},m_{\beta} \in \Rd\), and \(\Sigma_{\alpha},\,\Sigma_{\beta}\succ \textup{O}\). Define the Gaussian measures
\begin{equation}\label{eq:reference-measures}
    \alpha \Let c_\alpha \calN(m_\alpha,\Sigma_\alpha), \qquad
    \beta  \Let c_\beta  \calN(m_\beta,\Sigma_\beta),
\end{equation}
which serve as the reference measures throughout the article.

\subsection{UOT problem}\label{subsec:UOT:problem}
Given \eqref{eq:reference-measures}, we begin by specifying the problem data associated with the UOT problem. Let \(\pi \in \calM_2^{+}(\Rd \times \Rd)\) and suppose that \(E\) is a Borel subset of \(\Rd\). We define its marginals \(\pi_1\) and \(\pi_2\) by
\begin{align}\label{eq:marginal-def}
    \pi_1(E) \Let \pi(E \times \Rd), \quad
    \pi_2(E) \Let \pi(\Rd \times E) 
\end{align}
for every \(E \subset \Rd\). Let \(\gamma>0\) and define the ground cost \(\Rd \times \Rd \ni (x_1,x_2) \mapsto \ell(x_1,x_2)\coloneqq \|x_2-x_1\|^2\). With these ingredients, we consider the unbalanced optimal transport problem:

\begin{problembox}[\,(UOT problem)]\label{prob:uot}
Given the preceding problem data, find an optimal transport plan \(\pi \in \calM_2^{+}(\Rd \times \Rd)\) that solves the variational problem 
    \begin{equation}\label{eq:uot}
        \begin{aligned}
           &\inf_{\pi}  J_1(\pi)  \Let 
            \int_{\Rd}\int_{\Rd}  \hspace{-2mm} \|x_2-x_1\|^2 \,\rmd \pi(x_1,x_2)  + \gamma \KL{\pi_1}{\alpha}
            + \gamma \KL{\pi_2}{\beta},
        \end{aligned}
    \end{equation}
    where the marginals \(\pi_1\) and \(\pi_2\) are defined in \eqref{eq:marginal-def}.  
\end{problembox}
For the finiteness of the objective function in \eqref{eq:uot}, it is necessary that the marginals \(\pi_1\) and \(\pi_2\) be absolutely continuous with respect to Lebesgue measure; equivalently, they admit densities. This follows from the fact that \(\Sigma_{\alpha}\) and \(\Sigma_{\beta}\) are positive definite; see \eqref{eq:reference-measures}. Indeed, let \(\nu \Let \calN(m,\Sigma)\) with \(\Sigma \succ \textup{O}\), and suppose that
\(\widehat{\pi}\) is a finite nonnegative measure on \(\Rd\) which is not absolutely continuous with respect to Lebesgue measure \(\lambda\). Then there exists a measurable set \(A \subset \Rd\) such that
\(\lambda(A)=0\) and \(\widehat{\pi}(A)>0.\) Since \(\Sigma \succ \textup{O}\), the Gaussian measure \(\nu\) admits a Lebesgue density, and hence \(\nu\) is absolutely continuous with respect to \(\lambda\). Therefore, \(\lambda(A)=0\)
implies \(\nu(A)=0\). Thus \(\nu(A)=0\) while \(\pi(A)>0\), which shows that \(\widehat{\pi}\) is not absolutely continuous with respect to \(\nu\). Consequently, \(\KL{\widehat{\pi}}{\nu}=+\infty.\) Applying this observation to \(\pi_1\) and \(\pi_2\), we conclude that finiteness of \(\KL{\pi_1}{\alpha}\) and \(\KL{\pi_2}{\beta}\) requires absolute continuity of \(\pi_1,\pi_2\) with respect to \(\lambda.\)

\subsection{Dynamical formulation}\label{subsec:dynamical:formulation}
We now introduce a control-theoretic (dynamical) extension of Problem \ref{prob:uot}. Fix \(T,d,d'\in \bbN\); given the Gaussian reference measures in \eqref{eq:reference-measures}, we consider a discrete-time linear system
\begin{align}\label{eq:dynamics}
x_{t+1}=Ax_t+Bu_t \quad\text{for all }\, t=1,\dots,T-1,
\end{align}
with the problem data: 
\begin{enumerate}[label={(\textup{\ref{eq:dynamics}-\alph*})}, leftmargin=*, widest=b, align=left]
\item at time \(t\), \(x_t\in\R^d\) and \(u_t\in\R^{d'}\) represent the state and control variables, respectively;
\item \(A\in\R^{d\times d}\) and \(B\in\R^{d\times d'}\) are the system and control matrices; we assume that \(A\) is nonsingular. 
\end{enumerate}
For each \(t=1,\dots,T-1\), let \(U_t(\cdot\mid x)\) be a Borel probability measure on \(\R^{d'}\), measurable in \(x\), and interpreted as the conditional law of the control given the current state. We then define the state measures \(\{\pi_t\}_{t=1}^T\) recursively by
\begin{align}\label{eq:state:transition}
\pi_{t+1} = (Ax+Bu)_\#\bigl(\pi_t(\mathrm{d}x)\,U_t(\mathrm{d} u|x)\bigr),  
\end{align}
for \(t=1,\dots,T-1\). Since each \(U_t(\cdot|x)\) is a probability measure, the total mass is preserved along the dynamics, that is, \(\pi_t(\Rd)=\pi_1(\Rd)\) for \(t=1,\dots,T\). With \(\pi_1\in\calM_2^+(\Rd)\) and \(\mathbf{U}\Let\{U_t\}_{t=1}^{T-1}\), we define admissible set of controls by
\begin{align*}
\mathcal{U} \Let \left\{(\pi_1,\mathbf{U}) \,\middle\vert  \,
\begin{array}{@{}l@{}}
U_t(\cdot|x) \text{ is a Borel probability measure on }\R^{d'}\\ 
\text{which is measurable in }x,\;\eqref{eq:state:transition} \text{ holds},\\  \pi_t\in\calM_2^+(\Rd) \text{ for all }t=1,\dots,T \\ \text{and } \sum_{t=1}^{T-1}\int_{\Rd \times \R^{d'}}\|u\|^2\,
U_t(\rmd u|x)\pi_t(\rmd x)<\hspace{-1mm}+\infty
\end{array}
\right\}
\end{align*}
and consequently our unbalanced density control problem is: 
\begin{problembox}[\,(UDC problem)]\label{prob:udc}
 Let the horizon \(T\in\bbN\) be fixed and assume that the admissible set \(\mathcal{U}\) is nonempty. Given the preceding problem data, find an optimal initial measure \(\pi_1\in\calM_2^+(\Rd)\) and control policy \(\mathbf{U}\Let\{U_t\}_{t=1}^{T-1}\) that solves
    \begin{equation}\label{eq:udc}
        \begin{aligned}
            \inf_{(\pi_1,\mathbf{U})\in\mathcal{U}} \quad & \hspace{-3mm}J_2\bigl(\pi_1,\{U_t\}_{t=1}^{T-1}\bigr) \Let \exb{\sum_{t=1}^{T-1}\|u_t\|^2}  + \gamma \KL{\pi_1}{\alpha}+ \gamma \KL{\pi_T}{\beta}.  
        \end{aligned}
    \end{equation}   
\end{problembox} 
 In \eqref{eq:udc}, with slight abuse of notation, the expectation in the objective is understood as:
    \[
    \exb{\sum_{t=1}^{T-1}\|u_t\|^2} \Let
    \sum_{t=1}^{T-1}\int_{\Rd}\!\int_{\mathbb{R}^{d'}} \|u\|^2 \,\rmd U_t(u|x)\,\rmd \pi_t(x).
    \]
    We emphasize that $\pi_t$ is not necessarily a probability measure.

\section{Global Solutions to UOT}
\label{sec:main_result_uot}
We begin with a well-posedness result for Problem \ref{prob:uot}.

\begin{proposition}[Existence of a minimizer for UOT]\label{prop:uot-existence}
Consider Problem \ref{prob:uot} along with its data and notations. Let \(\alpha, \beta\) be as defined in \eqref{eq:reference-measures} and let \(\gamma>0\). Then, there exists an optimal transport plan \(\pi^*\in\calM^{+}_{2}(\Rd\times\Rd)\) that globally minimizes the UOT objective \eqref{eq:uot}.
\end{proposition}

\begin{proof}
We establish the existence of a minimizer by casting our formulation into the framework of \cite{Liero_2017} and verifying the conditions of \cite[Theorem 3.3]{Liero_2017}. To this end, first, we verify the basic topological and functional analytic hypotheses:
\begin{enumerate}[label={\textup{(P-\alph*)}}, leftmargin=*, widest=b, align=left]
\item \label{prop:step:1} The underlying space $\Rd$ with its usual norm $\norm{\cdot}$ is a complete, separable metric space, and thus a Hausdorff topological space.
\item \label{prop:step:2} The reference Gaussian measures $\alpha$ and $\beta$ are finite, non-negative Radon measures; see \eqref{eq:reference-measures}.
\item \label{prop:step:3} The quadratic transport cost $\Rd \times \Rd \ni (x_1,x_2) \mapsto \ell(x_1, x_2) \Let \norm{x_2 - x_1}^2$ is non-negative, proper, and lower semi-continuous. 
\end{enumerate}
\vspace{1mm}
Second, we check that the problem is strictly feasible. The KL divergence penalty corresponds to the logarithmic entropy function $s \mapsto F(s) \Let \gamma(s \log s - s + 1)$; because $F(0) = \gamma < +\infty$, the objective functional evaluated at the null transport plan $\pi = 0$ yields a finite cost \(J_1(0) = \gamma\alpha(\Rd) + \gamma \beta(\Rd) < +\infty,\) which indicates that the minimum is strictly feasible.

Define recession constant \cite[\S 2.1]{Sejourne2023} of \(F(\cdot)\) by $F'_{\infty} \Let \lim_{s \to +\infty} \frac{F(s)}{s}$, and notice that 
\begin{equation}
F'_{\infty} = \lim_{s \to +\infty} \frac{\gamma(s \log s - s + 1)}{s} = +\infty. \nonumber
\end{equation}
Thus, existence of a minimizer follows immediately from \cite[Theorem 3.3]{Liero_2017}.
\end{proof}


\subsection{Restriction to Gaussian marginals}
In this subsection, we develop the theoretical foundation for solving Problem \ref{prob:uot} via a finite-dimensional surrogate. The key ingredients are the following two propositions. Proposition \ref{prop:transport-lb} treats the quadratic transport term under fixed marginal moments, while Proposition \ref{prop:gauss-min-kl} identifies the Gaussian measure as the minimizer of the KL divergence to a Gaussian reference under fixed mass, mean, and covariance.

\begin{proposition}[Transport cost lower bound]\label{prop:transport-lb}
Consider Problem \ref{prob:uot} with its associated data and notations. Let $\pi_1,\pi_2\in\calM_2^{+}(\Rd)$  with the same mass $c>0$. For \(i \in \aset[]{1,2}\), let the mean and covariance of  $\bar \pi_i$ be $m_i \in \Rd$ and $\Sigma_i \succ \textup{O}$, respectively. Let \(\calC\) be a class of couplings, defined by
\begin{align*}
 \calC \Let \left\{ \pi \;\middle\vert\;  
\begin{array}{@{}l@{}}
\int_{\Rd \times \Rd} \rmd\pi =c,\,\, \bar \pi_{i}\mbox{ has mean }m_i, \mbox{covariance }\Sigma_{i}
\end{array}
\right\}.
\end{align*}
for each \(i \in \{1,2\}\). Then we have the following lower bound
\begin{equation}
    \begin{aligned}
        &\inf_{\pi\in\calC} \int_{\Rd}\int_{\Rd} \|x_2-x_1\|^2 \,\rmd \pi(x_1,x_2) \ge c\Big(\|m_2-m_1\|^2 +\tr\left(\Sigma_1+\Sigma_2-2(\Sigma_1^{1/2}\Sigma_2\Sigma_1^{1/2})^{1/2}\right)\Big). \nn
    \end{aligned}
\end{equation}
Moreover, if $\pi_1$ and $\pi_2$ are Gaussian measures, then equality holds, and an optimal coupling induced by the affine map
\begin{equation}\label{eq:uot-affine-map}
    \begin{aligned}
        &x_2 = T x_1 + \tau,\quad T \Let \Sigma_1^{-1/2}(\Sigma_1^{1/2}\Sigma_2\Sigma_1^{1/2})^{1/2}\Sigma_1^{-1/2},\,\,\text{and }\,\tau \Let m_2 - Tm_1.
    \end{aligned}
\end{equation}
\end{proposition}

\begin{proof}
Let \(\pi \in \calC\) be arbitrary, and define the normalized plan \(\bar{\pi} \Let \tfrac{1}{c}\pi\). Then \(\bar{\pi}\) is a probability measure on \(\Rd \times \Rd\), and its marginals \(\bar{\pi}_1\) and \(\bar{\pi}_2\) have means \(m_1,m_2\) and covariances \(\Sigma_1,\Sigma_2\), respectively. Thus \(\bar{\pi}\) belongs to the class of couplings
\begin{align*}
\bar{\calC} \Let \aset[\Big]{\bar{\pi} \suchthat \int_{\Rd \times \Rd}\hspace{-6mm}\rmd \bar{\pi} = 1,\;
\bar{\pi}_i \text{ has mean } m_i \text{ and covariance } \Sigma_i},
\end{align*}
where \(i\in \{1,2\}\). Therefore, by Theorem 2.1 of \cite{Gelbrich1990},
\begin{equation}
    \begin{aligned}
        & \int_{\Rd}\int_{\Rd} \|x_2-x_1\|^2 \,\rmd \bar{\pi}(x_1,x_2) 
        \ge \|m_2-m_1\|^2 +\tr\left(\Sigma_1+\Sigma_2-2(\Sigma_1^{1/2}\Sigma_2\Sigma_1^{1/2})^{1/2}\right). \nn 
    \end{aligned}
\end{equation}
Multiplying both sides by \(c\), and using \(\pi = c\bar{\pi}\), we obtain
\begin{equation*}
    \begin{aligned}
        &\quad\int_{\Rd}\int_{\Rd} \|x_2-x_1\|^2 \,\rmd \pi(x_1,x_2)  = c \int_{\Rd}\int_{\Rd} \|x_2-x_1\|^2 \,\rmd \bar{\pi}(x_1,x_2) \geq  \\
        & c\Bigl(\|m_2-m_1\|^2 +\tr\left(\Sigma_1+\Sigma_2-2(\Sigma_1^{1/2}\Sigma_2\Sigma_1^{1/2})^{1/2}\right)\Bigr).
    \end{aligned}
\end{equation*}
Since \(\pi \in \mathcal{C}\) was arbitrary, the desired lower bound follows.

Now suppose in addition that \(\pi_1\) and \(\pi_2\) are Gaussian measures. Then \(\bar{\pi}_1 \Let \calN(m_1,\Sigma_1)\) and \(\bar{\pi}_2 \Let \calN(m_2,\Sigma_2).\) By \cite[Proposition 7]{Clark1984}, the lower bound derived above is attained for Gaussian marginals. The optimality of the deterministic affine mapping then immediately follow from \cite[Example 1.7]{McCan1997}.
\end{proof}


\begin{proposition}[Gaussian minimizes KL]
\label{prop:gauss-min-kl}
Fix $c,c'>0$, $m,m'\in\Rd$, and $\Sigma,\Sigma'\in\mathbb{R}^{d \times d}$ with $\Sigma, \Sigma'\succ \textup{O}$. Let \(\pi\) be a finite nonnegative measure of mass \(c\) such that \(\bar\pi\) has mean \(m\) and covariance \(\Sigma\). Then we have the inequality
\begin{align}
    \KL{\pi}{c'\calN(m',\Sigma')}\geq\KL{c\calN(m,\Sigma)}{c'\calN(m',\Sigma')}. \nn
\end{align}
\end{proposition}

\begin{proof}
If \(\pi\) is not absolutely continuous with respect to \(c'\calN(m',\Sigma')\), then by the definition of the KL divergence, \(\KL{\pi}{c'\calN(m',\Sigma')} = +\infty,\) and the claim is immediate. Hence, it remains to consider the case when \(\pi\) is absolutely continuous with respect to \(c'\calN(m',\Sigma')\). Since \(c'>0\), this is equivalent to \(\bar{\pi}\) being absolutely continous with respect to \(\calN(m',\Sigma')\), so \(\bar{\pi}\) admits a Lebesgue density, say \(p\). We write
\begin{align}
\pi = c\,p(x)\rmd x
\quad \text{and} \quad c'\calN(m',\Sigma') = c'\,q(x)\rmd x,\nn
\end{align}
where \(q(x) \Let \calN(x;m',\Sigma')\); see notations in \S\ref{sec:introduction}. Then, using the expression of KL divergence for absolutely continuous unbalanced measures \cite{Sejourne2023} 
\begin{align}
\KL{\pi}{c'\calN(m',\Sigma')} &=
c \int_{\Rd} \log\left(\frac{p(x)}{q(x)}\right)p(x)\rmd x
+  c\log\left(\frac{c}{c'}\right) - c + c'.\nn
\end{align}
Which is \(\KL{\pi}{c'\calN(m',\Sigma')} = c\KL{\bar{\pi}}{\calN(m',\Sigma')}
+ c\log\left(\frac{c}{c'}\right) - c + c'.\)
We also have 
\begin{align}
& \KL{\bar{\pi}}{\calN(m',\Sigma')} = \int_{\Rd} p(x)\log\left(\frac{p(x)}{\calN(x;m',\Sigma')}\right)\rmd x \nn \\ & \hspace{5mm} = -H(p) - \int_{\Rd} p(x)\log \calN(x;m',\Sigma')\rmd x.\nn
\end{align}
Since, \(\log \calN(x;m',\Sigma') = -\tfrac{d}{2}\log(2\pi)-\tfrac{1}{2}\log\det\Sigma'-\tfrac{1}{2}(x-m')^\top(\Sigma')^{-1}(x-m'),\) the second term depends only on the mean and covariance of \(\bar{\pi}\), which are fixed to be \(m\) and \(\Sigma\). Hence, we arrive at 
\begin{align}
\KL{\pi}{c'\calN(m',\Sigma')} = -c\,H(p) + \mathrm{const.}, \nn
\end{align}
where \(\mathrm{const.}\) depends only on the fixed quantities \(c,c',m,m',\Sigma,\Sigma'\). From \cite[Theorem 8.6.5]{Cover2006}, it holds that \(H(p)\leq H(\calN(m,\Sigma))\) and consequently, \(\KL{\pi}{c'\calN(m',\Sigma')}
\ge \KL{c\calN(m,\Sigma)}{c'\calN(m',\Sigma')}.\) Our proof is complete.
\end{proof}

Leveraging Propositions \ref{prop:transport-lb} and \ref{prop:gauss-min-kl}, we now establish a central structural result for Problem \ref{prob:uot}. Even though the reference measures \(\alpha\) and \(\beta\) are Gaussian, it is not immediate that an optimizer of \eqref{eq:uot}, which is originally sought over a broad class of transport plans, should itself be Gaussian. The following theorem shows that, without loss of optimality, one may restrict attention to Gaussian transport plans.


\begin{theorem}[Optimality of Gaussian]\label{thm:gaussian-optimality}
Consider Problem \ref{prob:uot} with its associated data and notations and let \(\pi \in  \calM_2^{+}(\Rd \times \Rd)\) be a feasible solution to \eqref{eq:uot}.  Let $\pi_1,\pi_2\in\calM_2^{+}(\Rd)$ with the same mass $c>0$ and for \(i\in\{1,2\}\), let \(m_i\) and \(\Sigma_i\) denote the mean and covariance of the normalized marginal \(\bar{\pi}_i\), respectively. For \(i=1,2\), define the Gaussian measures \(\pi_i^G \Let c\,\calN(m_i,\Sigma_i)\) and let \(\pi^G\) be the coupling between \(\pi_1^G\) and \(\pi_2^G\) induced by the affine map in Proposition~\ref{prop:transport-lb}. Then \(J_1\bigl(\pi^G\bigr) \le J_1(\pi).\) Moreover, equality holds if \(\pi_i=c\,\calN(m_i,\Sigma_i)\) for \(i=1,2\) and \(\pi\) is the affine Gaussian coupling between these marginals.
\end{theorem}

\begin{proof}
Recall the expression of the objective function \(J_1(\cdot)\) in Problem \ref{prob:uot}. For convenience, we also define \(\mathsf{A} \Let \bigl(\|m_2-m_1\|^2+\tr\Sigma_1+\tr\Sigma_2- 2\tr\bigl((\Sigma_1^{1/2}\Sigma_2\Sigma_1^{1/2})^{1/2}\bigr)\bigr)\). We have the following chain of inequalities
   \begin{equation*}
        \begin{aligned}
            J_1(\pi) &\ge c\,\mathsf{A} + \gamma\KL{\pi_1}{\alpha}+\gamma\KL{\pi_2}{\beta}\ge c\,\mathsf{A} + \gamma\KL{\pi_{1}^{\rm G}}{\alpha}+\gamma\KL{\pi_{2}^{\rm G}}{\beta}\\
            &=\int_{\Rd}\int_{\Rd} \|x_2-x_1\|^2\,\rmd\pi^{\rm G} + \gamma\KL{\pi_{1}^{\rm G}}{\alpha}  +\gamma\KL{\pi_{2}^{\rm G}}{\beta} = J_1(\pi^{\rm G}).
        \end{aligned}
    \end{equation*}
We invoked Proposition~\ref{prop:transport-lb} for the first inequality and Proposition~\ref{prop:gauss-min-kl} for the second one.
Referring to Proposition \ref{prop:transport-lb} and \ref{prop:gauss-min-kl}, the equality holds throughout if $\pi_{i} = c\,\mathcal N(m_i,\Sigma_i)$ for $i=1,2$ and the coupling is the affine Gaussian one, i.e.\ the equality cases of Proposition \ref{prop:transport-lb} and \ref{prop:gauss-min-kl} hold simultaneously.
\end{proof}

\subsection{Optimization-based method for solving UOT}
Originally, Problem \ref{prob:uot} is formulated as a variational problem over transport plans \(\pi \in \calM_2^{+}(\Rd \times \Rd)\). As a result of Theorem~\ref{thm:gaussian-optimality}, Problem \ref{prob:uot} can be, without loss of optimality, reduced to a finite dimensional optimization. For $c,c'>0$, $m,m_1,m_2\in\Rd$, and $\Sigma,\Sigma_1,\Sigma_2\in\mathbb{R}^{d \times d}$ with $\Sigma,\Sigma_1, \Sigma_2\succ \textup{O}$, substituting \(\pi = c\ \calN(m, \Sigma)\) with corresponding marginal \(\pi_1 = c\ \calN(m_1, \Sigma_1)\) and \(\pi_2 = c\ \calN(m_2, \Sigma_2)\) to Problem \ref{prob:uot}, we obtain a finite dimensional optimization over \(c, m , \Sigma\), but since the optimal transport plan between \(\pi_1\) and \(\pi_2\) are induced by \eqref{eq:uot-affine-map}, the optimization problem is further reduced to the one over the tuple \(c, m_1, \Sigma_1, m_2, \Sigma_2\):
\begin{equation}\label{eq:finite-dim-uot}
    \begin{aligned}
        \min_{c,m_1,\Sigma_1,m_2,\Sigma_2} \quad c\,M(m_1,m_2) + c\,C(\Sigma_1,\Sigma_2) + \psi(c),
    \end{aligned}
\end{equation}
where various quantities in \eqref{eq:finite-dim-uot} are
\begin{align}\label{eq:uot-def-of-m-c-psi}
    &M(m_1,m_2) \Let \|m_2-m_1\|^2  \hspace{-1mm} + \hspace{-1mm}\frac{\gamma}{2}(m_1-m_\alpha)^\top\Sigma_\alpha^{-1}(m_1-m_\alpha)  \nonumber\\
    &\qquad\qquad\qquad +\frac{\gamma}{2}(m_2-m_\beta )^\top\Sigma_\beta ^{-1}(m_2-m_\beta )   \nonumber\\
    &C(\Sigma_1,\Sigma_2) \Let -2\tr\left(\sqrt{\Sigma_1^{1/2}\Sigma_2\Sigma_1^{1/2}}\right)  + \frac{\gamma}{2}\tr(\Sigma_\beta^{-1}\Sigma_{2}) -\frac{\gamma}{2}\log\det\Sigma_{2}+\tr\Sigma_{2}  \nonumber\\
    &\qquad\qquad\quad + \frac{\gamma}{2}\tr(\Sigma_\alpha^{-1}\Sigma_1)-\frac{\gamma}{2}\log\det\Sigma_1+\tr\Sigma_{1} \nonumber\\
    &\psi(c) \Let \frac{c\gamma}{2}(\log\det\Sigma_{\alpha} + \log\det\Sigma_{\beta} -2d)  + \gamma\phi_{c_\alpha}(c) + \gamma\phi_{c_\beta}(c) \\
    &\phi_{c'}(c)\Let c\log\!\frac{c}{c'}-c+c'.
\end{align}
Note that all functions defined in~\eqref{eq:uot-def-of-m-c-psi} are convex with respect to their arguments. However, since we have the multiplication between decision variables, ~\eqref{eq:finite-dim-uot} is not jointly convex. To this end, we proceed in the following manner; first, it can be seen that \eqref{eq:finite-dim-uot} is equivalent to
\begin{equation*}
    \begin{aligned}
        \min_{c}\min_{m_1,\Sigma_1,m_2,\Sigma_2} \quad c\,M(m_1,m_2) + c\,C(\Sigma_1,\Sigma_2) + \psi(c),
    \end{aligned}
\end{equation*}
which is also equivalent to the problem
\begin{equation}\label{eq:uot-reformulation}
    \begin{aligned}
        \hspace{-2mm}\min_{c}  \, c\Bigl(\,\min_{m_1,\Sigma_1,m_2,\Sigma_2}M(m_1,m_2) + C(\Sigma_1,\Sigma_2)\Bigr) + \psi(c).
    \end{aligned}
\end{equation}
Therefore, without loss of optimality, we can first perform the optimization with respect to \(m_1,\Sigma_1,m_2,\Sigma_2\), and perform the one with respect to \(c\) using the optimal value of the first subproblem. More precisely, denote by \(p^*\) the optimal value of the following subproblem:
\begin{equation}\label{eq:uot-subproblem}
    \begin{aligned}
        \min_{m_1,\Sigma_1,m_2,\Sigma_2}\,M(m_1,m_2) + C(\Sigma_1,\Sigma_2).
    \end{aligned}
\end{equation}
Then, \eqref{eq:uot-reformulation} can be written as 
\begin{equation}\label{eq:uot-c-optimization-prob}
    \begin{aligned}
        \min_{c} \,\, cp^* +\psi(c),
    \end{aligned}
\end{equation}
which is a scalar optimization problem, and we have the immediate consequences.
\begin{corollary}[Optimal mass for UOT problem]\label{cor:uot-optimal-c}
    Let \(p^*\) denote the optimal value of \eqref{eq:uot-subproblem}. Let \(L\Let\log\det \Sigma_\alpha+\log\det \Sigma_\beta-2d\), them the optimal mass for Problem \ref{prob:uot} is 
    \begin{align}\label{eq:cstar_uot}
        c^\ast \Let \sqrt{c_\alpha c_\beta}\exp\!\left(-\frac{p^*}{2\gamma}-\frac{L}{4}\right).
    \end{align}
\end{corollary}
\begin{proof}

Since \eqref{eq:uot-c-optimization-prob} is a scalar strictly convex optimization problem on \((0,+\infty)\), any stationary point is necessarily the unique global minimizer. A direct differentiation yields the results, and we omit the elementary calculation.
\end{proof}

In light of the above discussion, we present Algorithm~\ref{alg:UOT}, which computes a global optimum of Problem~\ref{prob:uot}.
\begin{algorithm2e}[t]
\caption{Unbalanced optimal transport between Gaussian measures}
\label{alg:UOT}
\DontPrintSemicolon
\KwIn{
Gaussian measures
$\alpha = c_\alpha \mathcal N(m_\alpha,\Sigma_\alpha)$,
$\beta = c_\beta \mathcal N(m_\beta,\Sigma_\beta)$,
and penalty parameter $\gamma>0$.
}
\KwOut{A globally optimal transport plan $\pi^\ast$ for Problem~\ref{prob:uot}.}

\textbf{Step 1:} Solve \eqref{eq:uot-subproblem} and find \((m_1^\ast,\Sigma_1^\ast,m_2^\ast,\Sigma_2^\ast)\) and let \(p^\ast \Let M(m_1^\ast,m_2^\ast)+C(\Sigma_1^\ast,\Sigma_2^\ast).\)

\textbf{Step 2:} Set \(L\Let \log\det\Sigma_\alpha+\log\det\Sigma_\beta-2d\)
and obtain the  optimal mass from Corollary \ref{cor:uot-optimal-c}.

\textbf{Step 3:} Construct the optimal Gaussian marginals: \(\pi_1^\ast \Let c^\ast\mathcal N(m_1^\ast,\Sigma_1^\ast)\) and \(\pi_2^\ast \Let c^\ast\mathcal N(m_2^\ast,\Sigma_2^\ast).\)

\textbf{Step 4:} Construct the optimal coupling by setting
\[
T^\ast \Let (\Sigma_1^\ast)^{-1/2}
\Bigl((\Sigma_1^\ast)^{1/2}\Sigma_2^\ast(\Sigma_1^\ast)^{1/2}\Bigr)^{1/2}
(\Sigma_1^\ast)^{-1/2},\] and \(\tau^\ast \Let m_2^\ast-T^\ast m_1^\ast.\)

Construct $\pi^\ast$ as the affine Gaussian coupling between
$\pi_1^\ast$ and $\pi_2^\ast$, using Step 4.

\Return{$\pi^\ast$}\;
\end{algorithm2e}

\begin{theorem}[Optimality of Algorithm~\ref{alg:UOT}]\label{thm:uot-alg-optimal}
    Algorithm~\ref{alg:UOT} computes a globally optimal solution to Problem \ref{prob:uot}.
\end{theorem}
\begin{proof}
By Theorem \ref{thm:gaussian-optimality}, Problem \ref{prob:uot} is reduced, without loss of optimality, to the finite-dimensional problem \eqref{eq:finite-dim-uot}, which is equivalent to \eqref{eq:uot-reformulation}. Step 1 of Algorithm \ref{alg:UOT} computes a global minimizer \((m_1^\ast,\Sigma_1^\ast,m_2^\ast,\Sigma_2^\ast)\)
of the convex problem \eqref{eq:uot-subproblem}, and let \(p^\ast\) denote its optimal value. Step 2 then computes the unique optimal mass \(c^\ast\) from
Corollary \ref{cor:uot-optimal-c}. Hence, the tuple \((c^\ast,m_1^\ast,\Sigma_1^\ast,m_2^\ast,\Sigma_2^\ast)\) is a global minimizer of \eqref{eq:finite-dim-uot}. Defining \(\pi_1^\ast \Let c^\ast \mathcal N(m_1^\ast,\Sigma_1^\ast)\) and \(\pi_2^\ast \Let c^\ast \mathcal N(m_2^\ast,\Sigma_2^\ast),\) and letting \(\pi^\ast\) be the affine Gaussian coupling between \(\pi_1^\ast\)
and \(\pi_2^\ast\) given by Proposition~\ref{prop:transport-lb}, Theorem~\ref{thm:gaussian-optimality}
implies that \(\pi^\ast\) is a globally optimal solution of Problem~\ref{prob:uot}.
\end{proof}

\subsubsection{Extension to entropic UOT}\label{subsec:entropic-uot}
In this brief subsection, we show that our previous results, can be extended for the well-known entropic UOT (EUOT) problem \cite{Janati2020}. EUOT is posed by the following variational problem:
 \begin{problembox}[\, (EUOT problem)]\label{prob:euot}
Let \(\sigma,\gamma>0\), and \(\alpha,\beta\) are as defined in \eqref{eq:reference-measures}. Find an optimal transport plan \(\pi \in \calM_2^{+}(\Rd \times \Rd)\) that solves: 
\begin{equation}
    \begin{aligned}
        &\inf_{\pi} J_3(\pi) \Let \int_{\Rd}\int_{\Rd} \norm{x_2-x_1}^2 \,\rmd\pi(x,y) + \\
        & \qquad\qquad \sigma \KL{\pi}{\alpha\otimes\beta} + \gamma \KL{\pi_1}{\alpha}
        + \gamma \KL{\pi_2}{\beta}. \nn
    \end{aligned}
\end{equation}     
\end{problembox}
The notation \(\alpha\otimes\beta\) is the standard product measure for given \(\alpha\) and \(\beta\). As in the UOT case, following the same steps as given in Proposition~\ref{prop:uot-existence}, one can show that Problem \ref{prob:euot} admits a minimizer. We therefore omit the details.

Moreover, similar to the UOT case Problem~\ref{prob:euot} can be reduced to a finite-dimensional optimization problem. 
\begin{theorem}\label{thm:euot-gauss-optimality}
Let \(\pi\) be any feasible solution to Problem \ref{prob:euot}. Denote mass of \(\pi\) by \(c\) and the mean and covariance of \(\bar{\pi}\) by
\begin{align}
      &m \Let \bbE_{(x,y)\sim\bar{\pi}}\begin{pmatrix}
            x \\ y
        \end{pmatrix} =\begin{pmatrix}
            m_{1} \\ m_{2}
        \end{pmatrix}, \,\,  \Sigma = \bbE_{(x,y)\sim\bar{\pi}}\left(\begin{pmatrix}
            x \\ y
        \end{pmatrix}-m\right)\left(\begin{pmatrix}
            x \\ y
        \end{pmatrix}-m\right)^\top =\begin{bmatrix}
            \Sigma_{1} & \Sigma_{3} \\ \Sigma_{3}^\top & \Sigma_{2}
        \end{bmatrix}. \nn
    \end{align}
   Then, for a Gaussian transport plan of the form \(\pi^{\rm G}\Let c\,\calN(m,\Sigma)\), it holds that \(J_3(\pi^{\rm G}) \leq J_3(\pi).\)
\end{theorem}
We skip the proof since it follows by applying the same steps given in Theorem \ref{thm:gaussian-optimality}. Exploiting this theorem, we reduce Problem~\ref{prob:euot} to the following finite dimensional optimization problem:
\begin{equation}\label{eq:finite-dim-euot}
    \begin{aligned}
        \min_{c,\Sigma_1, m_1,\Sigma_2,m_2,\Sigma_3} \hspace{-5mm}c\,\bar{M}(m_1,m_2) + c\,\bar{C}(\Sigma_1,\Sigma_2,\Sigma_3) + \bar{\psi}(c),
    \end{aligned}
\end{equation}
and in \eqref{eq:finite-dim-euot}, several terms are given by
\begin{align}
    &\bar{M}(m_1,m_2) \Let \|m_2-m_1\|^2   +\frac{\sigma+\gamma}{2}(m_1-m_\alpha)^\top\Sigma_\alpha^{-1}(m_1-m_\alpha)  \nonumber\\
    &\qquad\qquad\qquad +\frac{\sigma+\gamma}{2}(m_2-m_\beta )^\top\Sigma_\beta ^{-1}(m_2-m_\beta ),   \nonumber\\
    &\bar{C}(\Sigma_1, \Sigma_2, \Sigma_3) \Let \tr\Sigma_1 + \tr\Sigma_2 - 2\tr\Sigma_3 \nonumber\\
    & + \frac{(\sigma+\gamma)}{2}\bigl(\tr{\Sigma_{\alpha}^{-1}\Sigma_1} + \tr{\Sigma_{\beta}^{-1}\Sigma_2}- \log\det\Sigma_1\bigr) \nonumber\\
    &  - \frac{\sigma}{2}\log\det\left(\Sigma_2 - \Sigma_3^{\top}\Sigma_1^{-1}\Sigma_{3}\right) -\frac{\gamma}{2}\log\det\Sigma_2, \nonumber\\
    &\bar{\psi}(c)\Let c(\sigma + 2 \gamma) \left(\log c -1-\frac{d}{2}\right)  + \frac{c(\sigma+\gamma)}{2}\bigl(\log\det\Sigma_\alpha \nonumber\\
    & \qquad\quad + \log\det\Sigma_{\beta} - 2 \log c_{\alpha}c_{\beta}\bigr) + \sigma c_\alpha c_{\beta} + \gamma(c_{\alpha}+c_{\beta}). \nonumber
\end{align}
As in the UOT case, the problem is not jointly convex, but it decouples into the optimization over \(c\) and that over the remaining variables.

\begin{remark}
    In the EUOT case, the reduced finite-dimensional problem~\eqref{eq:finite-dim-euot} includes the off-diagonal block \(\Sigma_3\) of the covariance matrix of the transport plan as a decision variable. By contrast, in the UOT case, the reduced problem \eqref{eq:finite-dim-uot} involves only the diagonal blocks \(\Sigma_1\) and \(\Sigma_2\), corresponding to the marginal covariances. This distinction originates from the different nature of Theorems~\ref{thm:gaussian-optimality} and~\ref{thm:euot-gauss-optimality}. In the UOT case, Theorem~\ref{thm:gaussian-optimality} provides an explicit expression for the optimal off-diagonal block in terms of the diagonal blocks, so that the reduced problem can be written only in terms of the marginal covariances. In contrast, Theorem~\ref{thm:euot-gauss-optimality} reduces the infinite-dimensional EUOT problem to a finite-dimensional Gaussian optimization, but does not eliminate the off-diagonal block explicitly. Although one may formally express the stationary off-diagonal block in terms of the diagonal blocks through the first-order optimality condition, substituting this expression yields a nonconvex finite-dimensional problem, which is not tractable computationally. For this reason, we keep \(\Sigma_3\) as an explicit decision variable in the EUOT formulation.
\end{remark}
\section{Global Solutions to UDC}
\label{sec:main_result_udc}

In the previous section, we showed that the UOT (and the EUOT) problem, originally posed as a variational problem, can be transformed without loss of optimality to a finite-dimensional optimization over the first and second moments, and can therefore be solved globally via convex optimization. In this section, we establish an analogous result for the (density control) Problem \ref{prob:udc}.
\subsection{Gaussian reduction for UDC}
\label{subsection:gaussian_affine_reduction}
Recall the expression of the objective function \(J_2\bigl(\pi_1,\{U_t\}_{t=1}^{T-1}\bigr)\) in \eqref{eq:udc} and let \(\bigl(\pi_1,\{U_t\}_{t=1}^{T-1}\bigr)\) be a feasible solution to Problem \ref{prob:udc}. Define the total mass \(0<c \Let \int_{\R^d} \rmd\pi_1\) and, for each \(t=1,...,T-1\), define the normalized joint probability measure \(p_t\) on \(\Rd\times\R^{d'}\) by:
\begin{align*}
    p_t(\rmd x,\rmd u) \Let \frac{1}{c}U_t(\rmd u|x)\pi_t(\rmd x),
\end{align*}
where \(\pi_t\) is defined recursively by \eqref{eq:state:transition}. Denote the first and second moments by
\begin{align}\label{eq:moments_def}
    m_t & \Let \bbE_{(x_t,u_t)\sim p_t}[x_t], \quad\Sigma_t := {\rm Cov}_{(x_t,u_t)\sim p_t}(x_t), \nonumber\\
    \bar u_t & \Let \bbE_{(x_t,u_t)\sim p_t}[u_t], \quad \Sigma^u_t := {\rm Cov}_{(x_t,u_t)\sim p_t}(u_t),  \nonumber\\
    \Lambda_t & \Let {\rm Cov}_{(x_t,u_t)\sim p_t}(u_t,x_t)  = \bbE_{(x_t,u_t)\sim p_t}[\left(u_t - \bar{u}_t\right)\left(x_t - m_t\right)^\top].
\end{align}
Here \(m_t\) and \(\Sigma_t\) denote the mean and covariance of the normalized state distribution at time \(t\), equivalently of the state marginal of
\(p_t\). Thus \(m_t\), in the dynamical setting, plays the same role as the endpoint means \(m_1,m_2\) of the normalized marginals in the (UOT) Problem \ref{prob:uot}.

Using the moments in \eqref{eq:moments_def}, we construct a Gaussian initial distribution and a Gaussian conditional control law of the form
\begin{align}
    \pi_1^{\rm G} &:= c\,\calN(m_1,\Sigma_1), \label{eq:udc:gauss-initial-dist}\\
    U_t^{\rm G}(\cdot | x) &:= \calN\Bigl( \bar u_t + \Lambda_t \Sigma_t^{\dagger}(x-m_t),\; \Sigma^u_t - \Lambda_t \Sigma_t^{\dagger}\Lambda_t^\top\Bigr).\label{eq:gaussian_affine_policy}
\end{align}
Note that since
\begin{align*}
    \begin{bmatrix}
        \Sigma_t &\Lambda_t^\top \\ \Lambda_t & \Sigma_t^u
    \end{bmatrix}
\end{align*}
is the covariance matrix of \((x_t,u_t)\sim p_t\), it is positive semidefinite. Therefore, \(\Sigma^u_t - \Lambda_t \Sigma_t^{\dagger}\Lambda_t^\top\) is positive semidefinite, which follows from the generalized Schur complement for positive semidefinite block matrices \cite[Chapter 16, Theorem 16.1]{Gallier2011}. Consequently, \(U_t^{\rm G}(\cdot | x)\) is well-defined.
Let \(\pi_t^{\rm G}\) and \(p_t^{\rm G}\) denote the  measure and probability law induced by \eqref{eq:udc:gauss-initial-dist} and \eqref{eq:gaussian_affine_policy}:
\begin{align}
    & \pi_{t+1}^{\rm G} := (Ax+Bu)_{\#}(\pi_t^{\rm G}(\rmd x)U_t^{\rm G}(\rmd u |x)),\label{eq:udc:gauss-state-measure-def}\\
    & p_t^{\rm G}(\rmd x,\rmd u):= \frac{1}{c}U_t^{\rm G}(\rmd u|x)\pi_t^{\rm G}(\rmd x). \label{eq:udc:gauss-prob-def}
\end{align}
In the sequel, we prove that \(\bigl(\pi^{\rm G}_1,\{U_t^{\rm G}\}_{t=1}^{T-1}\bigr)\) attains equal or lower cost than the feasible pair \(\bigl(\pi_1,\{U_t\}_{t=1}^{T-1}\bigr)\), which justifies restricting Problem \ref{prob:udc} to Gaussian initial measures and affine-Gaussian control laws. Towards this end, we present the following supporting lemma.

\begin{lemma}\label{lemma:udc:moment-match}
Consider Problem \ref{prob:udc} together with its data and notation. Let \(\bigl(\pi_1,\{U_t\}_{t=1}^{T-1}\bigr)\) be any feasible solution to Problem \ref{prob:udc}, and let the associated moments \((m_t,\Sigma_t,\bar u_t,\Sigma_t^u,\Lambda_t)\) be as defined in \eqref{eq:moments_def}. Define the Gaussian initial measure \(\pi_1^{\rm G}\) and the Gaussian-affine control laws \(\{U_t^{\rm G}\}_{t=1}^{T-1}\) by \eqref{eq:udc:gauss-initial-dist} and \eqref{eq:gaussian_affine_policy}, and let \(\{\pi_t^{\rm G}\}_{t=1}^T\) and \(\{p_t^{\rm G}\}_{t=1}^{T-1}\) be the induced state measures and normalized joint laws defined in \eqref{eq:udc:gauss-state-measure-def} and \eqref{eq:udc:gauss-prob-def}, respectively. Then, for each \(t=1,\dots,T-1\), the joint law \(p_t^{\rm G}\) has the
same first and second moments as \(p_t\), i.e.,
\begin{align}
    \mathbb{E}_{(x_t,u_t)\sim p_t^{\rm G}}
    \begin{pmatrix} x_t\\u_t \end{pmatrix} &= \begin{pmatrix}
    m_t\\ \bar u_t \end{pmatrix}, \label{eq:udc:lem:moment-match-mean}\\[0.5em] {\rm Cov}_{(x_t,u_t)\sim p_t^{\rm G}}
    \begin{pmatrix} x_t\\ u_t \end{pmatrix} &= \begin{bmatrix}
    \Sigma_t & \Lambda_t^\top\\ \Lambda_t & \Sigma_t^u
    \end{bmatrix}. \label{eq:udc:lem:moment-match-cov}
\end{align}
In particular, we have
\begin{align}
    & \mathbb{E}_{\pi_t^{\rm G}}[x_t] = m_t,\,
    {\rm Cov}_{\pi_t^{\rm G}}(x_t) = \Sigma_t, 
    \mathbb{E}_{p_t^{\rm G}}[u_t] = \bar u_t, \nn \\ &
    {\rm Cov}_{p_t^{\rm G}}(u_t) = \Sigma_t^u, \,
     {\rm Cov}_{p_t^{\rm G}}(u_t,x_t) = \Lambda_t. \label{eq:udc:lem:moment-match}
\end{align}
\end{lemma}

Note that as defined in \eqref{eq:moments_def}, all the quantities on the right-hand side of \eqref{eq:udc:lem:moment-match} are computed with respect to \(p_t\), while those on the left-hand side are computed with respect to \(p_t^{\rm G}\).
\begin{proof}[Proof of Lemma~\ref{lemma:udc:moment-match}]
    In this proof, expectation and covariance are all computed with respect to \(p_t^{\rm G}\). For readability, we drop subscripts. 
    First, we prove that 
    \begin{equation}\label{eq:udc:lem:state-control-mean-equality}
        \begin{aligned}
            &\bbE[x_t] = m_t \text{ and }\bbE[u_t] = \bar{u}_t 
        \end{aligned}
    \end{equation}
    We proceed via induction. For \(t=1\), it holds by definition \eqref{eq:udc:gauss-initial-dist} that
    \begin{align}\label{eq:udc:lem:initial-mean-equality}
        \bbE[x_1] = m_1
    \end{align}
    By the tower property, we obtain
    \begin{equation}\label{eq:udc:lem:bar_u_initial}
        \begin{aligned}
            \bbE[u_1]& = \bbE\left[\bbE[u_1 |x_1]\right]  = \bbE[\bar u_1 + \Lambda_1 \Sigma_1^{\dagger}(x_1-m_1)] \\
            & = \bar{u}_1 + \Lambda_1 \Sigma_1^{\dagger}\left(\bbE[x_1] - m_1\right)  = \bar{u}_1
        \end{aligned}
    \end{equation}
    Note that we invoked \eqref{eq:udc:lem:initial-mean-equality} in the last equality.
    Suppose \eqref{eq:udc:lem:state-control-mean-equality} holds. Then, since \(m_t, \bar{u}_t\) satisfy \(m_{t+1} = Am_t + B\bar{u}_t\) from the dynamics constraint, we obtain
    \begin{align*}
        \bbE[x_{t+1}] &= A\bbE[x_t] + B \bbE[u_t]  = Am_t + B\bar{u}_t = m_{t+1}
    \end{align*}
    Employing similar steps as in \eqref{eq:udc:lem:bar_u_initial}, we obtain \(\bbE[u_{t+1}] = \bar{u}_{t+1}\), which completes the proof of \eqref{eq:udc:lem:state-control-mean-equality}. We now show
    \begin{equation}\label{eq:udc:lem:cov-equalities}
        \begin{aligned}
            {\rm Cov}(x_t) = \Sigma_t,\, {\rm Cov}(u_t, x_t) = \Lambda_t,\, \text{and }{\rm Cov}(u_t) = \Sigma_t^u.
        \end{aligned}
    \end{equation}
    We will again proceed via induction. For \(t=1\), it trivially holds by the definition \eqref{eq:udc:gauss-initial-dist} that:
    \begin{align}\label{eq:udc:lem:initial-cov-equality}
        {\rm Cov}(x_1) = \Sigma_1
    \end{align} 
    By the tower property, we obtain:
    \begin{align*}
        &{\rm Cov}(u_1, x_1) = \bbE\left[\bbE\left[(u_1-\bar{u}_1)|x_1\right]\left(x_1-m_1\right)^\top\right] \\
        & = \bbE\left[\left(\Lambda_1\Sigma_1^\dagger(x_1 - m_1)\right)\left(x_1 - m_1\right)^\top\right] \\
        & = \Lambda_1\Sigma_1^\dagger\bbE\left[\left(x_1 - m_1\right)\left(x_1 - m_1\right)^\top\right] = \Lambda_1
    \end{align*}
    In the last equality, we invoked \cite[Theorem 1]{Albert1969}.
    From the law of total covariance, it holds that 
    \begin{equation}\label{eq:udc:lem:cov_lambda_initial}
        \begin{aligned}
            {\rm Cov}(u_1) = &\bbE[{\rm Cov}(u_1|x_1)] +{\rm Cov}\left(\bbE[u_1|x_1]\right)
        \end{aligned}
    \end{equation}
    Substituting the definition \eqref{eq:udc:gauss-initial-dist}, we obtain:
        \begin{align}\label{eq:udc:lem:cov-u-initial}
            &{\rm Cov}(u_1)=  \Sigma^u_1 - \Lambda_1 \Sigma_1^{\dagger}\Lambda_1^\top + {\rm Cov}\left(\bar u_1 + \Lambda_1 \Sigma_1^{\dagger}(x-m_1)\right) \nn
            \\ &= \Sigma^u_1 - \Lambda_1 \Sigma_1^{\dagger}\Lambda_1^\top +  \Lambda_1 \Sigma_1^{\dagger}{\rm Cov}\left(x-m_1\right)  \Sigma_1^{\dagger} \Lambda_1^\top = \Sigma_1^u
        \end{align} 
    Note that in the last equality, we employed \eqref{eq:udc:lem:initial-mean-equality} and \eqref{eq:udc:lem:initial-cov-equality}. Suppose \eqref{eq:udc:lem:cov-equalities} holds. Then, since \(\Sigma_t, \Sigma_t^{u}, \) satisfy the covariance recursion \(\Sigma_{t+1} = A\Sigma_t A^\top + A \Lambda_t^\top B^\top + B\Lambda_t A^\top + B\Sigma_t^u B^\top\), it holds that
    \begin{align*}
        &{\rm Cov}[x_{t+1}]  = A{\rm Cov}(x_t)A^\top  + A {\rm Cov}(u_t,x_t)^\top B^\top + B{\rm Cov}( u_t,x_t) B^\top + B{\rm Cov}(u_t)B^\top \\
        & = A\Sigma_t A^\top +A \Lambda_t^\top B^\top + B\Lambda_t A^\top + B\Sigma_t^u B^\top = \Sigma_{t+1}
    \end{align*} 
    Applying the same computation as \eqref{eq:udc:lem:cov_lambda_initial} with the assumption that \eqref{eq:udc:lem:cov-equalities} holds, we can prove \({\rm Cov}(u_{t+1}, x_{t+1}) = \Lambda_{t+1}\). Likewise, applying the same computation as \eqref{eq:udc:lem:cov-u-initial} to \({\rm Cov}(u_{t+1})\) with the mean correspondence \eqref{eq:udc:lem:state-control-mean-equality}, which we have already proven, we can show \({\rm Cov}(u_{t+1}) = \Sigma_{t+1}^{u}\), which completes the proof. 
\end{proof}

\begin{theorem}[Gaussian reduction for UDC]\label{thm:udc_gaussian_restriction}
    Let $(\pi_1,\{U_t\}_{t=1}^{T-1})$ be any feasible solution to the (UDC) Problem~\eqref{prob:udc}. As Lemma~\ref{lemma:udc:moment-match}, let $(\pi_1^G,\{U_t^G\}_{t=1}^{T-1})$ be the Gaussian solution constructed from the first and second moments of $(\pi_1,\{U_t\}_{t=1}^{T-1})$. Then, it holds that
    \[
    J_2\bigl(\pi_1,\{U_t\}_{t=1}^{T-1}\bigr)\ \ge\ J_2\bigl(\pi_1^{\rm G},\{U_t^{\rm G}\}_{t=1}^{T-1}\bigr).
    \]
\end{theorem}
\begin{proof}
    First, we justify the following formula:
    \begin{align}\label{eq:udc-control-cost-equality}
        \bbE_{u_t\sim cp_t}\left[\sum_{t=1}^{T-1}\|u_t\|^2\right] = \bbE_{u_t\sim cp_t^{\rm G}}\left[\sum_{t=1}^{T-1}\|u_t\|^2\right]
    \end{align}
    From elementary computations, we obtain the following decomposition:
    \begin{align*}
        \bbE_{u_t\sim cp_t}\left[\sum_{t=1}^{T-1}\|u_t\|^2\right] = \sum_{t=1}^{T-1}c \tr (\Sigma_t^u + \bar u_t \bar u_t^\top)
    \end{align*}
    which implies that the quadratic control cost depends only on the first/second moments of control policy. From
    Lemma~\ref{lemma:udc:moment-match}, the first/second moments of control input by the solution \(\bigl(\pi_1,\{U_t\}_{t=1}^{T-1}\bigr)\) match that of \(\bigl(\pi_1^{\rm G},\{U_t^{\rm G}\}_{t=1}^{T-1}\bigr)\). Therefore, the control costs remain unchanged by the replacement. Moreover, as shown in Proposition~\ref{prop:gauss-min-kl}, among all measures with prescribed mass, mean, and covariance, the Gaussian measure minimizes the KL divergence to a Gaussian reference. Also, Lemma~\ref{lemma:udc:moment-match} proves that \(\pi_T\) and \(\pi_{T}^{\rm G}\) have the same mass, mean, and covariance. Therefore, the KL penalty terms are minimized by the replacement. The proof is complete.
\end{proof}

As a consequence, we may restrict attention to Gaussian state measures and \emph{Gaussian--affine feedback policies} of the form
\[
u_t \Let K_t(x_t-m_t)+v_t+w_t \quad\text{for } w_t\sim \calN(0,\Sigma^u_t).
\]
Under the above restriction, the UDC problem is equivalently written as the following optimization over
$c$, the mean trajectory $\{m_t\}_{t=1}^{T}$, the covariance trajectory $\{\Sigma_t\}_{t=1}^{T}$, and the affine-feedback parameters $\{v_t,K_t,\Sigma^u_t\}_{t=1}^{T-1}$
\begin{subequations}\label{eq:udc-gauss-reduced-problem}
    \begin{align}
        &\inf_{c,m_t,\Sigma_t,v_t,K_t,\Sigma^u_t} \,\,  c\Biggl[ \sum_{t=1}^{T-1}\Bigl(\|v_t\|^2 + \tr(K_t\Sigma_t K_t^\top) + \tr(\Sigma^u_t)\Bigr) + \gamma M(m_1,m_T)  \nonumber\\
        & \qquad\qquad\qquad\qquad+ \gamma S(\Sigma_1,\Sigma_T) \Biggr] + \gamma \psi(c) \label{eq:udc_reduced}\\
        &\qquad\,\, \text{ s.t.}\,\,\, m_{t+1}=A m_t + B v_t,\nonumber\\
        & \qquad\qquad\,\,\Sigma_{t+1}=(A+BK_t)\Sigma_t(A+BK_t)^\top + B\Sigma^u_tB^\top \label{eq:udc-cov-recursion}\\
        &\qquad\qquad\,\,\text{for }t=1,\dots,T-1,\nonumber
    \end{align}
\end{subequations}
 where \(M\) and \(S\) in \eqref{eq:udc-gauss-reduced-problem} are given by
\begin{align*}
    &M(m_1,m_T) \Let \frac12(m_1-m_\alpha)^\top \Sigma_\alpha^{-1}(m_1-m_\alpha) +\frac12(m_T-m_\beta)^\top \Sigma_\beta^{-1}(m_T-m_\beta), \\
    &S(\Sigma_1,\Sigma_T)  \Let \frac12\Bigl\{\tr(\Sigma_\alpha^{-1}\Sigma_1)+\tr(\Sigma_\beta^{-1}\Sigma_T)-\log\det\Sigma_1-\log\det\Sigma_T\Bigr\},
\end{align*}
and $c\mapsto \psi(c)$ is the same as the UOT case~\eqref{eq:uot-def-of-m-c-psi}. As in the UOT case, we separate \eqref{eq:udc-gauss-reduced-problem} into the optimization over \(c\) and the one over other remaining variables, \(m_t, \Sigma_t, v_t, K_t, \Sigma_t^{u}\) to derive a computationally tractable reformulation. The latter problem is still a nonconvex optimization problem due to the bilinear terms involving \(K_t\) and \(\Sigma_t\) in \eqref{eq:udc_reduced}--\eqref{eq:udc-cov-recursion}. To obtain a \emph{convex formulation} without loss of optimality, we apply the standard covariance-steering lifting \cite{Balci2022}
\begin{align}\label{eq:udc:change-variable}
    S_t \Let K_t\Sigma_t,\quad Y_t \Let K_t\Sigma_t K_t^\top +\Sigma_t^u.
\end{align}
Under this change of variables, the covariance recursion becomes affine, and
the relation between \((S_t,Y_t)\) and \((K_t,\Sigma_t^u)\) is expressed exactly by the linear matrix inequality
\begin{align}\label{udc:reform-LMI-const}
\begin{bmatrix}
Y_t & S_t\\ S_t^\top & \Sigma_t
\end{bmatrix}\succeq \textup{O}. 
\end{align}
Hence, we obtain the following convex subproblem:
\begin{subequations}\label{eq:sdp_subproblem-convex-reform}
    \begin{align}
        &\inf_{m_t,\Sigma_t,v_t,S_t,Y_t} \,\,  \sum_{t=1}^{T-1}\Bigl(\|v_t\|^2 + \tr(Y_t)\Bigr) + \gamma M(m_1,m_T) + \gamma S(\Sigma_1,\Sigma_T) \nonumber\\
        &\qquad\,\, \text{s.t.}\,\,\, m_{t+1}=A m_t + B v_t,\label{udc:subproblem-LMI-const} \\
        & \qquad\qquad\,\Sigma_{t+1} = A\Sigma_t A^\top + BS_tA^\top + AS_t^\top B^\top + BY_tB^\top \label{udc:subproblem-covariance-dynamics}\\
        &\qquad\qquad\,\dmat{Y_{t}}{S_{t}}{S_{t}^{\top}}{\Sigma_{t}} \succeq \textup{O} \text{ for }t=1,\dots,T-1.\nn
    \end{align}
\end{subequations}
For notational convenience, let us define
\[
z \Let \Bigl(\{m_t,\Sigma_t\}_{t=1}^{T},\{v_t,S_t,Y_t\}_{t=1}^{T-1}\Bigr),
\]
and consider the feasible set for the problem \eqref{eq:sdp_subproblem-convex-reform}
\begin{align*}
    \mathcal{F}\Let \left\{ z \;\middle\vert  \;
    \begin{array}{@{}l@{}}
        \eqref{udc:subproblem-LMI-const},\,\eqref{udc:subproblem-covariance-dynamics} \text{ and \eqref{udc:reform-LMI-const} hold for } t=1,\dots,T-1
    \end{array}
    \right\}.
\end{align*}
Also define the mapping 
\begin{align}
\mathcal{F} \ni  z \mapsto f(z) &\Let \sum_{t=1}^{T-1}\bigl(\|v_t\|^2+\tr(Y_t)\bigr)+\gamma M(m_1,m_T) +\gamma S(\Sigma_1,\Sigma_T) \in \R. \nn
\end{align}
With these preceding notations, \eqref{eq:udc-gauss-reduced-problem} can be written as
\begin{align}\label{eq:udc:total:problem}
      \inf_{c> 0,\; z\in\mathcal{F}}\, \bigl(c\,f(z)+\psi(c)\bigr).
\end{align}
We are now ready to study the well-posedness of the problems \eqref{eq:udc-gauss-reduced-problem}, \eqref{eq:sdp_subproblem-convex-reform}, and \eqref{eq:udc:total:problem}.

\subsection{Existence of solutions}
By Theorem~\ref{thm:udc_gaussian_restriction}, the (UDC) Problem~\ref{prob:udc} can be restricted, without loss of optimality, to Gaussian initial measures and affine-Gaussian control laws. Furthermore, the reparametrization~\eqref{eq:udc:change-variable} yields the exact reformulation~\eqref{eq:sdp_subproblem-convex-reform}. Therefore, to establish existence of a minimizer for Problem~\ref{prob:udc}, it suffices to prove existence of a minimizer for~\eqref{eq:udc:total:problem}.  The present subsection formalizes this fact.

\begin{proposition}[Existence of a minimizer for the reduced UDC problem]\label{prop:existence-reduced-udc}
Consider the problems \eqref{eq:udc-gauss-reduced-problem}, \eqref{eq:sdp_subproblem-convex-reform}, and \eqref{eq:udc:total:problem} along with their data and notations. Then the convex subproblem \eqref{eq:sdp_subproblem-convex-reform} admits a minimizer, and consequently the problem \eqref{eq:udc:total:problem}, equivalently \eqref{eq:udc-gauss-reduced-problem}, admits a minimizer.
\end{proposition}



\begin{proof}
Consider the feasible set \(\fsblset\) for \eqref{eq:sdp_subproblem-convex-reform}, and observe that \(\fsblset\) is closed. Indeed, \eqref{udc:subproblem-LMI-const} and \eqref{udc:subproblem-covariance-dynamics} are affine equality constraints, hence define closed sets, and \eqref{udc:reform-LMI-const} is a
positive-semidefinite constraint, hence also defines a closed set. Therefore \(\fsblset\), being a finite intersection of closed sets, is closed. In \eqref{eq:udc:total:problem}, we first consider the inner problem \(p^\star:=\inf_{z\in\mathcal{F}} f(z)\) and we check that this problem has a finite-valued feasible point: it suffices to choose \(m_t=0,v_t=0,\Sigma_1=I,S_t=0,\) and \(Y_t=0\) for all $t$. Then~\eqref{udc:subproblem-covariance-dynamics} yields \(\Sigma_{t}=A^{t-1}(A^\top)^{t-1}\). Since \(\text{det}\ A\neq 0\), \(-\log\det\Sigma_T\) is finite, therefore the objective function is also finite.

Moreover, \(f(\cdot)\) is sequentially lower semicontinuous on \(\fsblset\). The quadratic, trace, and mean terms are continuous. For the endpoint covariance term, define \(X_1 \Let \Sigma_\alpha^{-1/2}\Sigma_1\Sigma_\alpha^{-1/2},\) and \(X_T \Let \Sigma_\beta^{-1/2}\Sigma_T\Sigma_\beta^{-1/2}.\)
Then, up to an additive constant depending only on \(\Sigma_\alpha\) and \(\Sigma_\beta\), the covariance part of \(S(\Sigma_1,\Sigma_T)\) is \(\bigl(\tr(X_1)-\log\det X_1\bigr) + \bigl(\tr(X_T)-\log\det X_T\bigr).\) Extending \(X \mapsto \tr(X)-\log\det X\) by \(+\infty\) on singular matrices, this function is lower semicontinuous. Hence the endpoint covariance term is lower semicontinuous \cite[\S 1.2.1]{ref:santambrogio2023course}.

We now show that \(f\) is sequentially coercive \cite[Proposition 1.1]{Ambrosio1988} on \(\fsblset\) via the following steps: 
\begin{itemize}[leftmargin=*, label=\(\circ\)]
\item  For \(C\in\mathbb R\), define the sublevel set
\[
\fsblset_C:=\{z\in\fsblset\;|\;f(z)\le C\}.
\]
We claim that \(\fsblset_C\) is compact. Since \(f\) is lower semicontinuous and \(\fsblset\) is closed, \(\fsblset_C\) is closed. It remains to prove boundedness. Let \(z\in\fsblset_C\). Since
\[
M(m_1,m_T)\ge 0,\qquad \|v_t\|^2\ge 0,\qquad \operatorname{tr}(Y_t)\ge 0
\]
whenever \((35)\) holds, and since \(S(\Sigma_1,\Sigma_T)\) is bounded below by a constant depending only on \(\Sigma_\alpha\) and \(\Sigma_\beta\), it follows that \(\sum_{t=1}^{T-1}\|v_t\|^2\), \(\sum_{t=1}^{T-1}\operatorname{tr}(Y_t)\), \(M(m_1,m_T)\),and \(S(\Sigma_1,\Sigma_T)\) are all bounded above on \(\fsblset_C\). Hence \(\{v_t\}\) is bounded for each \(t\), and \(m_1,m_T\) are bounded. Also, from \eqref{udc:reform-LMI-const}, we have \(Y_t\succeq \textup{O}\), and therefore \(\|Y_t\|\le \operatorname{tr}(Y_t).\) Thus \(\{Y_t\}\) is bounded for each \(t\).

\item Boundedness of \(S(\Sigma_1,\Sigma_T)\) implies that \(\Sigma_1\) and \(\Sigma_T\) are bounded and uniformly positive definite. Indeed, if \(\lambda\) denotes an
eigenvalue of \(X_1\) or \(X_T\), then \(\lambda-\log\lambda\to+\infty\) as \(\lambda\downarrow 0\) or \(\lambda\uparrow+\infty,\) so the eigenvalues remain in a compact subinterval of \((0,+\infty)\). Then \eqref{udc:subproblem-LMI-const} implies boundedness of \(\{m_t\}\) for all \(t\).

\item Next, multiplying the inequality \eqref{udc:reform-LMI-const} on the left and right by \(\begin{bmatrix}B & -A\end{bmatrix}\) and its
transpose gives
\[
B S_t A^\top + A S_t^\top B^\top
\preceq
A\Sigma_t A^\top + B Y_t B^\top.
\]
Substituting this into \eqref{udc:subproblem-covariance-dynamics}, we obtain \(\Sigma_{t+1}
\preceq 2A\Sigma_t A^\top + 2B Y_t B^\top.\) Since \(\Sigma_1\) and \(Y_t\) are bounded, an induction argument shows that \(\{\Sigma_t\}\) is bounded for all \(t\).

\item Finally, from inequality \eqref{udc:reform-LMI-const} one has \(\|S_t\|^2\le \|Y_t\|\,\|\Sigma_t\|\), and consequently, \(\{S_t\}\) is bounded as well. Hence \(\fsblset_C\) is bounded, and therefore compact.
\end{itemize}
Thus \(f(\cdot)\) is sequentially coercive. Then together with the lower semicontinuity of \(f(\cdot)\) we conclude that, \(f\) attains its minimum on
\(\fsblset\). Therefore there exists \(z^\star\in\fsblset\) such that \(f(z^\star)=p^\star\) \cite[Proposition 1.2]{Ambrosio1988}.


For the optimization problem 
\[
\inf_{c> 0}\bigl(c\,p^\star+\psi(c)\bigr),
\]
the map \(c\mapsto c\,p^\star+\psi(c)\) is continuous on \([0,+\infty)\) and coercive as \(c\to+\infty\). Hence, there exists $R>0$ such that the infimum over $[0,+\infty)$ is equal to the minimum over the compact interval $[0,R]$. By the Weierstrass theorem, \cite[\S 1.2.1]{ref:santambrogio2023course} the infimum is attained at some $c^\star\ge 0$. Finally, for every $c\ge 0$ and every $z\in\mathcal{F}$
\begin{align*}
    c f(z)+\psi(c) &\ge c p^\star+\psi(c) \ge c^\star p^\star+\psi(c^\star) = c^\star f(z^\star)+\psi(c^\star).
\end{align*}
Hence, $(c^\star,z^\star)$ attains the infimum in~\eqref{eq:udc:total:problem}.
\end{proof}

Now, we can prove the following corollary, which ensures the existence of a solution for the UDC problem.

\begin{corollary}[Existence of a minimizer for the UDC problem]\label{cor:existence-udc}
   The Problem~\ref{prob:udc} admits an optimal solution.
\end{corollary}

\begin{proof}
By Proposition \ref{prop:existence-reduced-udc}, the problem \eqref{eq:udc:total:problem}, equivalently the Gaussian-reduced UDC problem \eqref{eq:udc-gauss-reduced-problem}, admits a minimizer. This minimizer determines a Gaussian initial measure and an affine-Gaussian control law. By Theorem \ref{thm:udc_gaussian_restriction}, such a Gaussian solution is optimal among all feasible solutions of Problem \ref{prob:udc}. Hence, Problem \ref{prob:udc} admits an optimal solution.
\end{proof}

\subsection{Optimization method for UDC and properties of optimal solution}

By the separation in \eqref{eq:udc:total:problem}, the optimization over the mass \(c\) can be carried out independently once the optimal value of the convex inner problem \eqref{eq:sdp_subproblem-convex-reform} is known. Let \(p^\star\) denote the optimal value of \eqref{eq:sdp_subproblem-convex-reform}. Then \eqref{eq:udc:total:problem} reduces to the scalar problem
\begin{align}
    \min_{c>0} \; c p^\star + \psi(c). \nn
\end{align}
Therefore, we obtain the following result, which is analogous to Corollary~\ref{cor:uot-optimal-c}. Since its proof follows the same argument, we omit the details. 
\begin{corollary}[Optimal mass for UDC problem]\label{cor:udc-optimal-c}
    Let \(p^\star\) the optimal value of \eqref{eq:sdp_subproblem-convex-reform} and let \(L:=\log\det \Sigma_\alpha+\log\det \Sigma_\beta-2d\). Then the optimal mass for UDC problem~\eqref{eq:udc} is 
    \begin{align}\label{eq:cstar_udc}
        c^\ast=\sqrt{c_\alpha c_\beta}\exp\!\left(-\frac{p^\star}{2\gamma}-\frac{L}{4}\right)
    \end{align}
and Algorithm~\ref{alg:UDC} computes a globally optimal solution to~\eqref{eq:udc}.
\end{corollary}


\begin{algorithm2e}[t]
\caption{Unbalanced Density Control between Gaussians.}
\label{alg:UDC}
\DontPrintSemicolon
\KwIn{Gaussians $\alpha=c_{\alpha}\mathcal N(m_\alpha,\Sigma_\alpha)$, $\beta=c_{\beta}\mathcal N(m_\beta,\Sigma_\beta)$; \(\gamma \geq 0\), \(T\in\bbN\), \(A\in\R^{d\times d}\), \(B\in\R^{d\times d'}\) }
\KwOut{Global Solution \(\pi_1^*\), \(\{U_t\}_{t=1}^{T-1}\).}
  \textbf{Step 1:} Covariance steering (SDP and QP) \\
  \quad Solve \eqref{eq:sdp_subproblem-convex-reform},\,
  \(p^{\star} \longleftarrow \text{Optimal Value of \eqref{eq:sdp_subproblem-convex-reform}}\) \\
  \textbf{Step 2:} Mass computation (no optimization)
  \begin{align*}
      c^* \longleftarrow &\sqrt{c_{\alpha}c_{\beta}}\exp\Bigl(-\frac{p^{\star}}{2\gamma}-\frac{L}{4}\Bigr), \\&\text{with }L:=\log\det \Sigma_\alpha+\log\det \Sigma_\beta-2d)
    \end{align*}
\end{algorithm2e}

While the SDP reformulation above is already sufficient to compute a globally optimal solution, we now give a structural sufficient condition for deterministic optimality. In particular, one may ask whether the optimal affine-Gaussian policy requires an additional control covariance term \(\Sigma_t^{u}\), or whether deterministic affine feedback already suffices. The next theorem shows that, if the optimal state covariance sequence is positive definite throughout the horizon, then the corresponding optimal affine-Gaussian policy is deterministic.


\begin{theorem}\label{thm:sigma-psd-and-deterministic-policy}
Consider the problem~\eqref{eq:sdp_subproblem-convex-reform}. If the optimal \(\Sigma_t\) is positive-definite for all \(t=1,...,T\), the optimal policy in~\eqref{eq:sdp_subproblem-convex-reform}  is deterministic; that is, the optimal solution satisfies
    \begin{align}\label{eq:sigma-t-u-zero}
        \Sigma_t^u = Y_t - S_t \Sigma_t^{-1} S_t^\top = \nullmat, \quad\mbox{for }  t=1,\ldots,T-1.
    \end{align}
\end{theorem}
\begin{proof}
To prove \eqref{eq:sigma-t-u-zero}, we consider the Karush-Kuhn-Tucker condition. Let us denote the Lagrange multiplier for the constraint \eqref{udc:subproblem-covariance-dynamics} by \(L_t\) and the one for \eqref{udc:subproblem-LMI-const} by \(M_t\). We utilize the following structure
    \begin{align*}
        M_t = \dmat{M_t^{00}}{M_t^{01}}{M_t^{10}}{M_t^{11}}.
    \end{align*}
    The complementary slackness condition of the LMI constraint is given by
    \begin{align}\label{eq:complementary-slackness}
        \dmat{M_t^{00}}{M_t^{01}}{M_t^{10}}{M_t^{11}}\dmat{Y_{t}}{S_{t}}{S_{t}^{\top}}{\Sigma_{t}}=\textup{O}.
    \end{align}
    From the stationary condition with respect to \(Y_t\) and \(S_t\), the following holds
    \begin{subequations}
        \begin{align}
            &B^{\top}L_t A + M_t^{01} = 0,\label{eq:stationary-cond-Y_t} \\
            &I - B^{\top} L_t B + M_t^{00} =0.\label{eq:stationary-cond-S_t}
        \end{align}
    \end{subequations}
    Under \(\Sigma_t\succ \textup{O}\), the following factorization holds
    \begin{equation}\label{eq:LMI-factorization}
        \begin{aligned}
            &\dmat{Y_{t}}{S_{t}}{S_{t}^{\top}}{\Sigma_{t}}  = \dmat{I}{S_{t}\Sigma_t^{-1}}{\textup{O}}{I}\dmat{Y_t - S_t\Sigma_t^{-1}}{\textup{O}}{\textup{O}}{\Sigma_t}\dmat{I}{\textup{O}}{\Sigma_t^{-1}}{I}
        \end{aligned}
    \end{equation}
    Plugging this to \eqref{eq:complementary-slackness}, we obtain
    \begin{align*}
        \dmat{M_t^{00}}{M_t^{01}}{M_t^{10}}{M_t^{11}}\dmat{I}{S_{t}\Sigma_t^{-1}}{\textup{O}}{I}\dmat{Y_t - S_t\Sigma_t^{-1}}{\textup{O}}{\textup{O}}{\Sigma_t} = \textup{O}
    \end{align*}
    Therefore, it holds that
    \begin{subequations}
        \begin{align}
            &M_t^{00}\left(Y_t - S_t\Sigma_t^{-1}S_t^\top\right) = \textup{O}\label{eq:m00} \\
            &(M_t^{01})^\top \left(Y_t - S_t\Sigma_t^{-1}S_t^\top\right) = \textup{O}\label{eq:m01}
        \end{align}
    \end{subequations}
    Combining \eqref{eq:stationary-cond-Y_t}, \eqref{eq:m01}, and the invertibility of \(A\), we get
    \begin{align}\label{eq:control-noise-zero-step}
        L_t B \left(Y_t - S_t \Sigma_t^{-1} S_t^\top\right) = \textup{O}.
    \end{align}
    From \eqref{eq:stationary-cond-S_t}, \eqref{eq:m00}, and \eqref{eq:control-noise-zero-step}, we finally obtain
    \begin{equation*}
        \begin{aligned}
            &\left(I - B^\top L_t B\right)\left(Y_t - S_t^\top \Sigma_t^{-1}S_t\right)  = Y_t - S_t^\top \Sigma_t^{-1}S_t =\textup{O}
        \end{aligned}
    \end{equation*}
    The proof is complete.
\end{proof}
This result is analogous to the one that deterministic affine policy is optimal for the UOT problem.



\section{MaxEnt UDC}
\label{sec:maxent_udc}
We conclude our theoretical developments by considering a maximum-entropy variant of the proposed UDC formulation. The maximum-entropy viewpoint \cite{Ito2023} augments the quadratic control objective with an entropy reward on the control policy, thereby encouraging randomized policies with higher dispersion while preserving the prescribed moment-level steering structure. In this section, we introduce a MaxEnt UDC problem, formulated as a UDC problem with an additional entropy-regularization term. We keep all the assumptions and data from Problem \ref{prob:udc} (with the obvious modifications)  and consider the following problem.
\begin{problembox}[\,(MaxEnt UDC)]\label{prob:maxent_udc} 
Let the horizon \(T\in\bbN\) be fixed and assume that the admissible set \(\mathcal{U}\) defined in \S\ref{subsec:dynamical:formulation} is nonempty. Let \(\eps>0\) and \(\gamma>0\); given the preceding problem data, find an optimal initial measure \(\pi_1\in\calM_2^+(\Rd)\) and control policy \(\{U_t\}_{t=1}^{T-1}\) that solves
    \begin{equation}
        \begin{aligned}
              & \hspace{-3mm}\inf_{\pi_1,\,\{U_t\}_{t=1}^{T-1}}   \exb{\sum_{t=1}^{T-1}\left(\|u_t\|^2 - \eps\ent{U_t(\cdot | x_t)}\right) } \nn  \\
            & \hspace{-2mm}+ \gamma \KL{\pi_1}{\alpha}+ \gamma \KL{\pi_T}{\beta}  =: J_{4}(\pi_1,\{U_t\}_{t=1}^{T-1})  
        \end{aligned}
    \end{equation}  
\end{problembox}
\subsection{Gaussian reduction for MaxEnt UDC}
Using our previous developments we develop a Gaussian reduction for Problem \ref{prob:maxent_udc}.

\begin{theorem}[Gaussian reduction for MaxEnt UDC]\label{thm:gauss-optimal-maxent-udc}
    Let \((\pi_1,\{U_t\}_{t=1}^{T-1})\) be any feasible solution to Problem \ref{prob:maxent_udc}. As per Lemma~\ref{lemma:udc:moment-match}, let \((\pi_1^{\rm G},\{U_t^{\rm G}\}_{t=1}^{T-1})\) be the Gaussian solution constructed from the first and second moment of \((\pi_1,\{U_t\}_{t=1}^{T-1})\). Then the following inequality holds
    \begin{align*}
        J_{4}\bigl(\pi_1,\{U_t\}_{t=1}^{T-1}\bigr) \geq J_{4}\bigl(\pi_1^{\rm G},\{U_t^{\rm G}\}_{t=1}^{T-1}\bigr).
    \end{align*}
\end{theorem}
\begin{proof}
    Note that \(J_4 = J_2 - \eps\bbE\left[\sum_{t=1}^{T-1}\ent{U_t(\cdot | x_t)}\right]\). Also, from Theorem~\ref{thm:udc_gaussian_restriction}, we have \(J_{2}\bigl(\pi_1,\{U_t\}_{t=1}^{T-1}\bigr) \geq J_{2}\bigl(\pi_1^{\rm G},\{U_t^{\rm G}\}_{t=1}^{T-1}\bigr)\). Therefore, it suffices to verify that
    \begin{align*}
        \bbE\left[\sum_{t=1}^{T-1}\ent{U_t(\cdot | x_t)}\right] \leq \bbE\left[\sum_{t=1}^{T-1}\ent{U_t^{\rm G}(\cdot | x_t)}\right].
    \end{align*}
    For each \(t=1,...,T-1\), let us the control policy \(\bar{U}_t(u|x) := U_t(u+\Lambda_t\Sigma_t^\dagger x|x),\)
    where \(\Lambda_t\) is defined in \eqref{eq:udc:lem:moment-match-cov}. Then it holds that
    \begin{align*}
        \bbE_{x_t\sim\pi_t}\left[\ent{U_t(\cdot | x_t)}\right] &=  \bbE_{x_t\sim\pi_t}\left[\ent{\bar{U}_t(\cdot | x_t)}\right] \leq \bbE_{x_t\sim\pi_t}\left[\ent{\bar{U}_t}\right]  = c\, \ent{\bar{U}_t}.
    \end{align*}
    The equality holds since \(\bar{U}_t(\cdot|x)\) is a shift of \(U_t(\cdot|x)\) under fixed variable \(x\), and the inequality is valid since the removing the conditioning does not increase entropy. Also, \(\bar{U}_t\) and \(U_t^{\rm G}(\cdot|x)\) have the same covariance and it is well-known that Gaussian attains maximum entropy under fixed covariance, thus  \(c\,\ent{\bar{U}_t}\leq c\,\ent{U_t^{\rm G}(\cdot|x)}  = \bbE_{x_t\sim\pi_t^{\rm G}}\left[\ent{U_t^{\rm G}(\cdot | x_t)}\right].\) Therefore, we obtain \(\bbE_{x_t\sim\pi_t}\left[\ent{U_t(\cdot | x_t)}\right] \leq \bbE_{x_t\sim\pi_t^{\rm G}}\left[\ent{U_t^{\rm G}(\cdot | x_t)}\right].\) Summing up both sides from \(t=1\) to \(T-1\), we obtain the desired inequality, which finishes the proof.
\end{proof}

\subsection{Equivalent reformulation of MaxEnt UDC}
Employing Theorem~\ref{thm:gauss-optimal-maxent-udc} we can transform Problem \ref{prob:maxent_udc} into the following finite dimensional optimization problem:
\begin{equation}\label{eq:maxent-gauss}
    \begin{aligned}
        &\inf_{c,m_t,\Sigma_t,v_t,K_t,\Sigma^u_t} \,\,  c\Biggl[ \sum_{t=1}^{T-1}\Bigl(\|v_t\|^2 + \tr(K_t\Sigma_t K_t^\top) + \tr(\Sigma^u_t) - \frac{\eps}{2} \log\det(\Sigma_t^u)\Bigr)   \\
        & \qquad\qquad\qquad\qquad + \gamma M(m_1,m_T) + \gamma S(\Sigma_1,\Sigma_T) \Biggr] + \gamma \psi(c) \\
        &\qquad\,\, \text{s.t.}\,\,\, m_{t+1}=A m_t + B v_t,\\
        & \qquad\qquad\,\,\Sigma_{t+1}=(A+BK_t)\Sigma_t(A+BK_t)^\top + B\Sigma^u_tB^\top \,\text{ for }t=1,\dots,T-1.
    \end{aligned}
\end{equation}
Since the objective function and the covariance recursion contains bilinear terms, we perform a change of variables by setting \(S_t \Let K_t\Sigma_t\) and \(Y_t \Let K_t \Sigma_t K_t^\top.\) Through this manipulation, \eqref{eq:maxent-gauss} admits the following equivalent reformulation
\begin{subequations}\label{eq:maxent-reform}
    \begin{align}
        &\inf_{c,m_t,\Sigma_t,\Sigma_t^{u},v_t,S_t,Y_t} \,\,  c\Biggl[\sum_{t=1}^{T-1}\Bigl(\|v_t\|^2 + \tr(Y_t) + \tr(\Sigma_t^{u}) -\eps \log\det\Sigma_t^u\Bigr) + \gamma M(m_1,m_T)  \nonumber\\
        &\qquad\qquad\qquad\qquad\,+ \gamma S(\Sigma_1,\Sigma_T) \Biggr] + \gamma \psi(c)\nonumber\\
        &\qquad\,\, \text{s.t.}\,\,\, m_{t+1}=A m_t + B v_t,\nonumber\\
        & \qquad\qquad\,\Sigma_{t+1} = A\Sigma_t A^\top + BS_tA^\top + AS_t^\top B^\top+ B(Y_t+\Sigma_t^u)B^\top \label{eq:maxent:subproblem-covariance-dynamics}\\
        &\qquad\qquad\, Y_t = S_t \Sigma_t^{-1} S_t^\top \,\text{for }t=1,\dots,T-1.\label{eq:maxent:invertibility-cond}
    \end{align}
\end{subequations}
The constraint \eqref{eq:maxent:invertibility-cond} ensures the invertibility of the transformation. However, \eqref{eq:maxent-reform} is still a nonconvex problem due to \eqref{eq:maxent:invertibility-cond}. Now, consider the relaxation of \eqref{eq:maxent:invertibility-cond} via the condition \(Y_t \succeq S_t \Sigma_t^{-1} S_t^\top.\) Observe that, the preceding inequality can be written as a convex formulation via Schur complement as:
\begin{align}\label{eq:maxent:shur}
    \dmat{Y_{t}}{S_{t}}{S_{t}^{\top}}{\Sigma_{t}} \succeq \textup{O}.
\end{align}
Finally, by replacing \eqref{eq:maxent:invertibility-cond} with \eqref{eq:maxent:shur}, we obtain a relaxation of \eqref{eq:maxent-reform}: 
\begin{subequations}\label{eq:maxent-relaxation}
    \begin{align}
        &\inf_{c,m_t,\Sigma_t,\Sigma_t^{u},v_t,S_t,Y_t} \,\,  c\Biggl[\sum_{t=1}^{T-1}\Bigl(\|v_t\|^2 + \tr(Y_t) +  +\tr(\Sigma_t^{u}) -\eps \log\det\Sigma_t^u\Bigr) + \gamma M(m_1,m_T)  \nonumber\\
        &\qquad\qquad\qquad\qquad\,+ \gamma S(\Sigma_1,\Sigma_T) \Biggr] + \gamma\psi(c)\nonumber\\
        &\qquad\,\, \text{s.t.}\,\,\, m_{t+1}=A m_t + B v_t,\label{eq:maxent-udc-relaxed-mean-dynamics}\\
        & \qquad\qquad\,\Sigma_{t+1} = A\Sigma_t A^\top + BS_tA^\top + AS_t^\top B^\top  + B(Y_t+\Sigma_t^u)B^\top \label{eq:maxent-udc-relaxed-cov-dynamics}\\
        &\qquad\qquad\, \dmat{Y_{t}}{S_{t}}{S_{t}^{\top}}{\Sigma_{t}} \succeq \textup{O} \,\text{for } t=1,\dots,T-1 \label{eq:maxent-udc-relaxed-LMI-constraint}.
    \end{align}
\end{subequations}
We are now ready to show the existence results.

\subsection{Existence of minimizer}\label{subsec:existence:MaxEnt}
To complete the existence argument for Problem \ref{prob:maxent_udc}, it remains to show that the relaxed finite-dimensional problem~\eqref{eq:maxent-relaxation} admits a minimizer. Once this is established, we can combine it with the tightness of the relaxation and the Gaussian reduction result in Theorem \ref{thm:maxent-relax-tight} ahead to conclude existence of a minimizer for the original MaxEnt UDC problem.
\begin{proposition}[Existence of a minimizer for the relaxed MaxEnt UDC problem]\label{prop:maxent-udc-relaxed-existence}
The relaxed version \eqref{eq:maxent-relaxation} of Problem \ref{prob:maxent_udc} admits a solution.
\end{proposition}
We skip the proof of Proposition~\ref{prop:maxent-udc-relaxed-existence} since it follows along similar lines given in Proposition \ref{prop:existence-reduced-udc}. We may therefore apply the following tightness argument to an actual optimizer of \eqref{eq:maxent-relaxation}.

\begin{theorem}\label{thm:maxent-relax-tight}
    Any optimal solution to \eqref{eq:maxent-relaxation} satisfies \(Y_t = S_t \Sigma_t^{-1}S_t^\top\). Namely, relaxation from \eqref{eq:maxent-reform} to \eqref{eq:maxent-relaxation} is tight.
\end{theorem}
To prove this theorem, we introduce the following lemma.

\begin{lemma}\label{lem:maxent-eigen}
Let \(A\) and \(B\) be symmetric nonzero matrices. If \(A \succ \textup{O}\) and \(B \succeq \textup{O}\), then \(\det(A+B) > \det(A).\)
\end{lemma}

\begin{proof}
   Since \(A\) is positive definite, its inverse matrix and square root matrices are well-defined. Therefore, we obtain
    \begin{align*}
        \det(A+B) &= \det\left(A^{\frac{1}{2}}(I + A^{-\frac{1}{2}}BA^{-\frac{1}{2}})A^{\frac{1}{2}}\right) \\
        & = \det A^{\frac{1}{2}}\det(I + A^{-\frac{1}{2}}BA^{-\frac{1}{2}}) \det A^{\frac{1}{2}} \\
        & =  \det A \det(I + A^{-\frac{1}{2}}BA^{-\frac{1}{2}}).
    \end{align*}
    Set \(C:=A^{-\frac{1}{2}}BA^{-\frac{1}{2}}\). Since \(B\) is positive semidefinite, the positive semidefiniteness of \(C\) follows. Also, since \(B\) is not a zero matrix and \(A\) is invertible, \(C\) is not a zero matrix. Therefore, letting \(\lambda_1,...,\lambda_n\) be eigenvalues of \(C\), it holds that \(\lambda_i\geq 0\) for all \(i\in\{1,...,n\}\) and there exists \(j\in\{1,...,n\}\) such that \(\lambda_j>0\). Since eigenvalues of \(I+C\) are \(1+\lambda_1,...,1+\lambda_n\), the determinant of \(I+C\) is \(\det(I+C) = \prod^n_{i=1}(1+\lambda_i) > 1.\) Therefore, we obtain the desired inequality.
\end{proof}

\begin{proof}[Proof of Theorem~\ref{thm:maxent-relax-tight}]
    While \(c, (m_t,v_t)_t\) are part of the decision variables in \eqref{eq:maxent-reform} and \eqref{eq:maxent-relaxation}, the claim under consideration depends only on the covariance-related variables. Accordingly, we focus on \((S_t,Y_t,\Sigma_t,\Sigma_t^u)_t\).
    The proof proceeds via contradiction. Let \((S_t,Y_t,\Sigma_t,\Sigma_t^u)_t\) be an optimal solution to \eqref{eq:maxent-relaxation}. For \(t\in\bbN\), let \(\Delta_t:=Y_t - S_t\Sigma_t^{-1}S_t^{\top}\). Suppose there exists \(\tau\in\{1,...,T-1\}\) such that \(\Delta_\tau\neq O\). Construct a new feasible solution \((\widetilde{S}_t,\widetilde{Y}_t,\widetilde{\Sigma}_t,\widetilde{\Sigma}_t^u)_t\) as follows:
    \begin{subequations}
        \begin{align*}
            & \widetilde{S}_t \Let S_t   \qquad\quad\quad\,\,\mbox{for }t=\{1,..,T-1\} \\
            &\widetilde{Y}_t \Let \begin{cases}
                Y_\tau - \Delta_\tau  &\mbox{if }t=\tau \\
                Y_t & \mbox{otherwise}  
            \end{cases} \\
            &\widetilde{\Sigma}_t \Let \Sigma_t   \qquad\quad\quad\,\,\mbox{for }t=\{1,..,T-1\} \\
            &\widetilde{\Sigma}_t^u \Let \begin{cases}
                \Sigma_\tau^u + \Delta_\tau  &\mbox{if }t=\tau \\
                \Sigma_t^u & \mbox{otherwise}  
            \end{cases}
        \end{align*}
    \end{subequations}
    It is straightforward to see that \((\widetilde{S}_t,\widetilde{Y}_t,\widetilde{\Sigma}_t,\widetilde{\Sigma}_t^u)_t\) is feasible for \eqref{eq:maxent-relaxation}. Let \(J(\cdot) := \tr(Y_t) + \tr(\Sigma_t^{u}) -\eps \log\det\Sigma_t^u\). Then
    \begin{align*}
        &J\bigl((S_t,Y_t,\Sigma_t,\Sigma_t^u)_t\bigr) - J\bigl((\widetilde{S}_t,\widetilde{Y}_t,\widetilde{\Sigma}_t,\widetilde{\Sigma}_t^u)_t\bigr) = \tr(Y_\tau) - \tr(\widetilde{Y_\tau}) + \tr(\Sigma_\tau^u) - \tr(\widetilde{\Sigma}_{\tau}^u) \\ 
        &\quad- \eps\log\det\Sigma_{\tau}^u + \eps\log\det\widetilde{\Sigma}_\tau^u  = \tr(Y_\tau) - \tr(Y_\tau- \Delta_{\tau}) + \tr(\Sigma_\tau^u)  \\
        & \quad - \tr(\Sigma_{\tau}^u + \Delta_{\tau}) +\eps\left(\log\det(\Sigma_{\tau}^u+\Delta_{\tau}) - \log\det\Sigma_{\tau}^u\right) = \eps\left(\log\det(\Sigma_{\tau}^u+\Delta_{\tau}) - \log\det\Sigma_{\tau}^u\right)
        \end{align*}
        For finiteness of the objective function, it is necessary to hold that \(\Sigma_\tau^u\succ \textup{O}\). Applying Lemma~\ref{lem:maxent-eigen} for \(A \Let \Sigma_{\tau}^u\) and \(B \Let \Delta_{\tau}\), we obtain \(\det(\Sigma_{\tau}^u+\Delta_{\tau}) > \det\Sigma_{\tau}^u,\)
        which, in turn, implies that 
        \[J\bigl((S_t,Y_t,\Sigma_t,\Sigma_t^u)_t\bigr) >J\bigl((\widetilde{S}_t,\widetilde{Y}_t,\widetilde{\Sigma}_t,\widetilde{\Sigma}_t^u)_t\bigr).\]
        This contradicts the optimality of \((S_t,Y_t,\Sigma_t,\Sigma_t^u)_t\).
\end{proof}
Building on the arguments so far, the proof of the following corollary is immediate.
\begin{corollary}[Existence of a minimizer for the MaxEnt UDC problem]
Problem~\ref{prob:maxent_udc} admits a solution.
\end{corollary}
\begin{proof}
By Proposition~\ref{prop:maxent-udc-relaxed-existence}, the relaxed problem~\eqref{eq:maxent-relaxation} admits a minimizer. By Theorem~\ref{thm:maxent-relax-tight}, every minimizer of~\eqref{eq:maxent-relaxation} satisfies
\[
Y_t = S_t \Sigma_t^{-1} S_t^\top \quad\text{for } t=1,\dots,T-1,
\]
and therefore is also a minimizer of the exact finite-dimensional formulation~\eqref{eq:maxent-reform}. By the equivalence between \eqref{eq:maxent-gauss} and \eqref{eq:maxent-reform}, this minimizer yields a Gaussian initial measure and an affine-Gaussian control law solving the Gaussian-reduced problem~\eqref{eq:maxent-gauss}. Finally, by Theorem~\ref{thm:gauss-optimal-maxent-udc}, such a Gaussian solution is optimal among all feasible solutions of Problem~\ref{prob:maxent_udc}. Hence, Problem~\ref{prob:maxent_udc} admits an optimal solution.
\end{proof}

\section{Numerical experiments}\label{sec:NumExp}
In this section, we present the result of numerical simulation and observe the behavior of the optimal solutions which are computed by proposed algorithms. We implemented all the simulations on Python(v3.12.2) and used \texttt{cvxpy} library.

\subsection{UOT}\label{subsec:simulation-uot} We begin by illustrating our results for UOT problem. To this end, consider the two cases:
\begin{enumerate}[leftmargin=*]
\item (Case 1) \emph{Unbalanced}: \(\alpha_1 \Let 1.0 \,\mathcal{N}(-1.0, 0.9^2)\) and \(\beta_1 \Let 0.6 \,\mathcal{N}(1.2, 0.6^2)\).
 \item (Case 2) \emph{Balanced}: \(\alpha_2 \Let 1.0 \,\mathcal{N}(-1.0, 0.9^2)\) and \(\beta_2 \Let 1.0 \,\mathcal{N}(1.2, 0.6^2).\)
\end{enumerate}
In Problem \ref{prob:uot}, for each value of \(\gamma\), we apply Algorithm~\ref{alg:UOT} to the two one-dimensional test cases specified above and compute the corresponding optimal transport plan \(\pi^*\), its marginals \(\pi_1\) and \(\pi_2\), and the optimal transported mass \(c^*\). The resulting optimal marginals are compared with the prescribed reference measures in Fig.~\ref{fig:uot-marginal-unbal_bal} while the corresponding values of \(c^*\) are reported in Table~\ref{tab:optimization_stats_fw}. The top subfig. in Fig.~\ref{fig:uot-marginal-unbal_bal} corresponds to the unbalanced case \((\alpha,\beta)=(\alpha_1,\beta_1)\), where the two reference measures have different total masses, whereas the bottom subfig. in Fig.~\ref{fig:uot-marginal-unbal_bal} shows the balanced case \((\alpha,\beta)=(\alpha_2,\beta_2)\), where the reference measures have equal mass. Recall that in the UOT objective \eqref{eq:uot}, the parameter \(\gamma\) determines the weight of the KL penalties that enforce proximity of the marginals to the reference measures. Therefore, as \(\gamma\) increases, deviations of \(\pi_1\) and \(\pi_2\) from \(\alpha\) and \(\beta\) become increasingly expensive, and one expects the optimal marginals to move closer to the reference measures. This trend is clearly visible in both subfigures in Fig.~\ref{fig:uot-marginal-unbal_bal}: \(\pi_1\) and \(\pi_2\) depicted by the orange dashed curves, approach \(\alpha\) and \(\beta\), shown by the blue curves, as \(\gamma\) increases.
\begin{table}[b]
    \centering
    \begin{tblr}{
        colspec = {l c c},
        cells = {font=\normalsize},
    }
    \hline[2pt]
     \SetRow{azure9}
            \(\gamma\) & \(c^* \,\)(1. Unbalanced)  & \(c^* \,\)(2. Balanced)\\
            \hline
            0.2 & 0.289 & 0.373 \\
            \hline
            1.0 &  0.368 &  0.474 \\
            \hline
            10.0 &  0.634 & 0.819 \\
            \hline
            30.0 &  0.718 &  0.927 \\
     \hline[2pt]
    \end{tblr}
    \vspace{2.5mm}
    \caption{The optimal \(c^*\) for each \(\gamma\) in UOT Simulation.}
    \label{tab:optimization_stats_fw}
\end{table}

\begin{figure*}[t]
\centering
\includegraphics[scale=0.3]{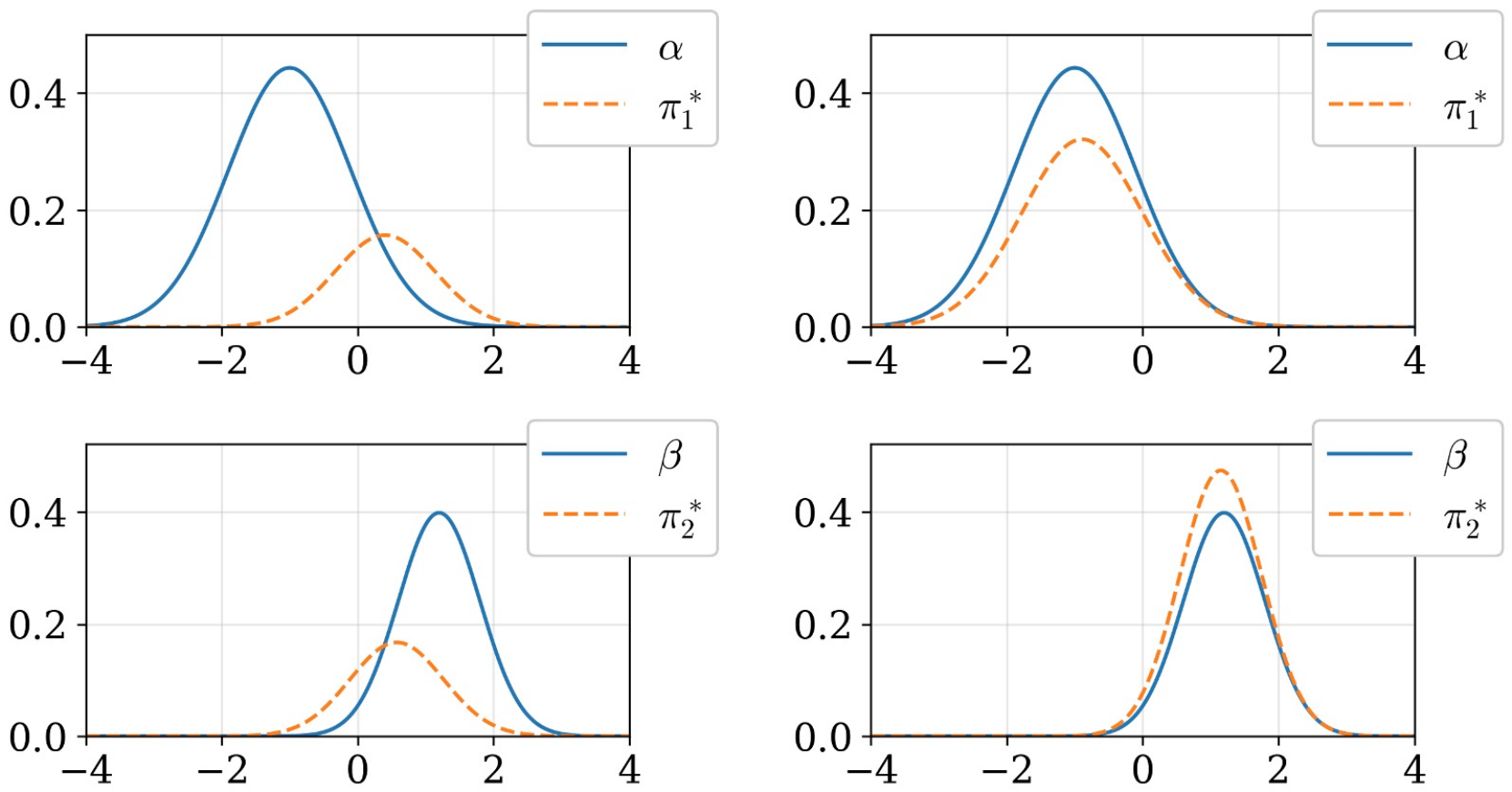}
\includegraphics[scale=0.33]{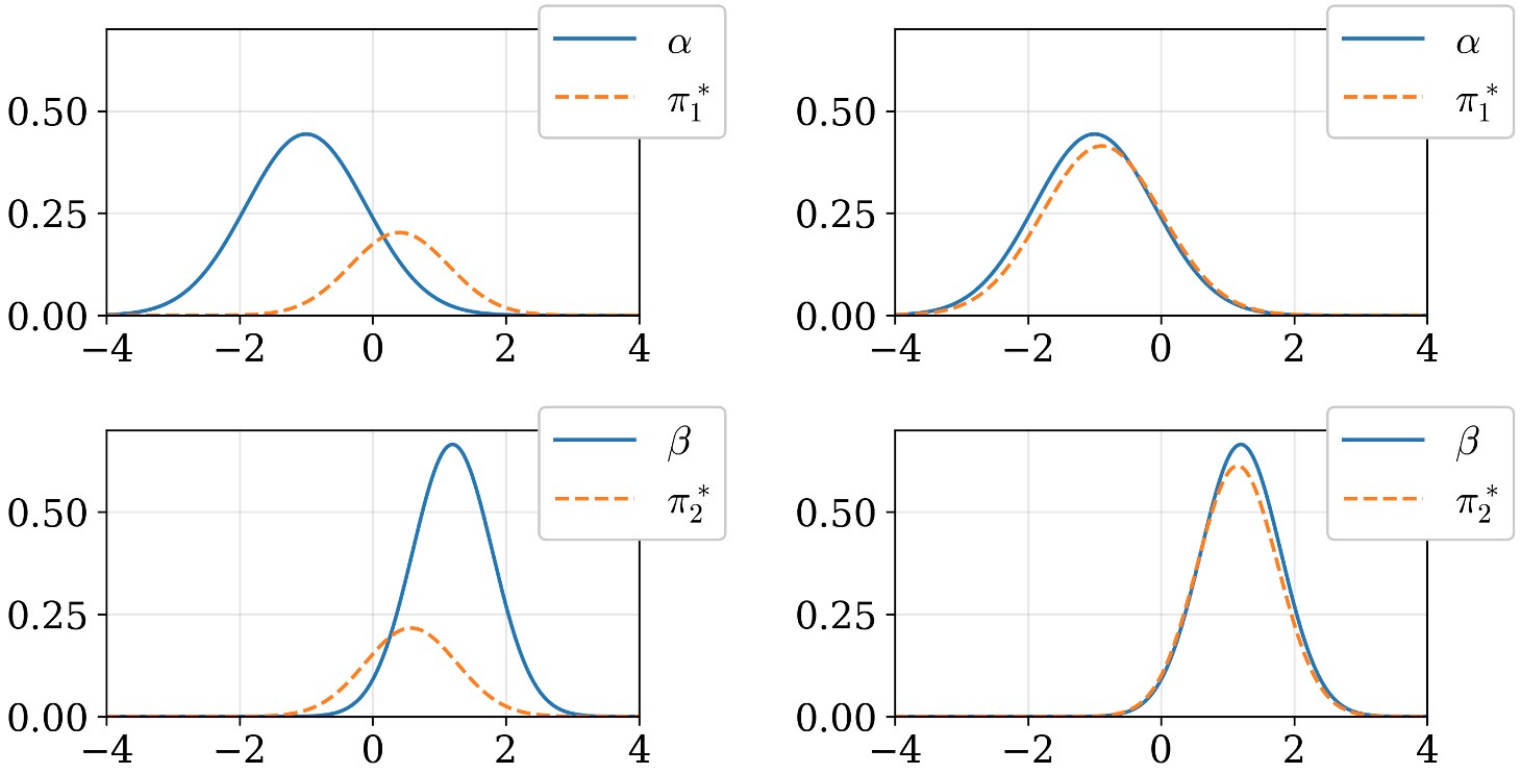}
\caption{The optimal marginal distribution in Case 1 (top) and Case 2 (bottom), where \(\gamma=0.2\) and \(\gamma=30\), respectively. The dotted orange line indicates the optimal marginal distribution, and the blue line specifies the given reference measure.}
\label{fig:uot-marginal-unbal_bal}
\end{figure*}

\begin{figure*}[t]
\centering
\setlength{\arrayrulewidth}{0.35pt}
\includegraphics[scale=0.25]{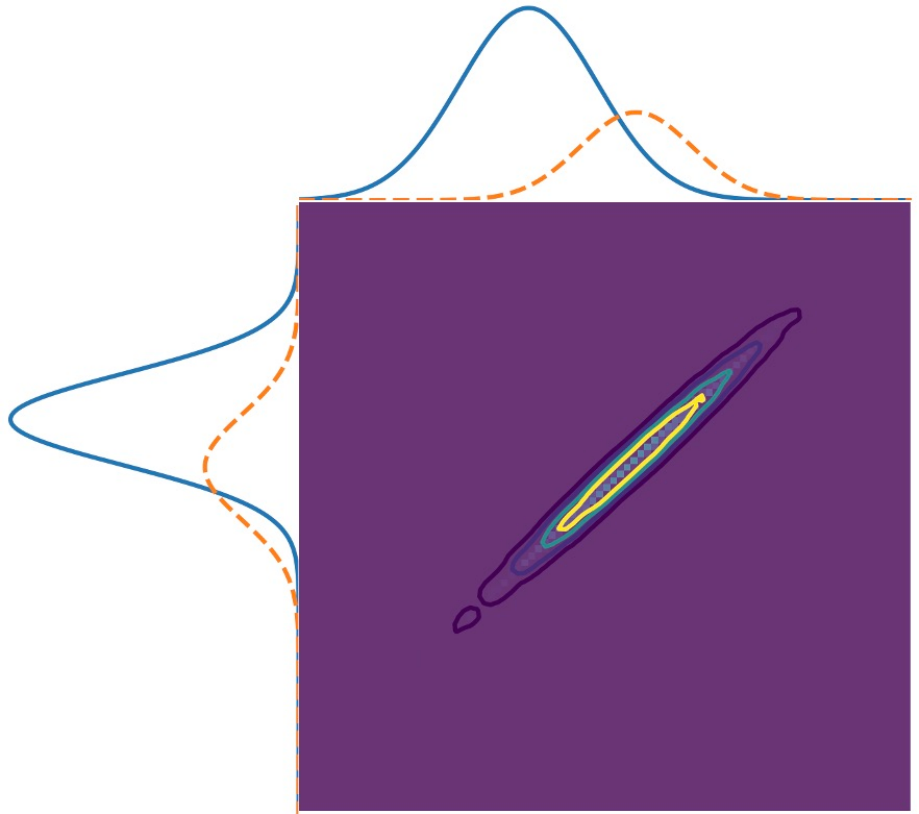}
\includegraphics[scale=0.25]{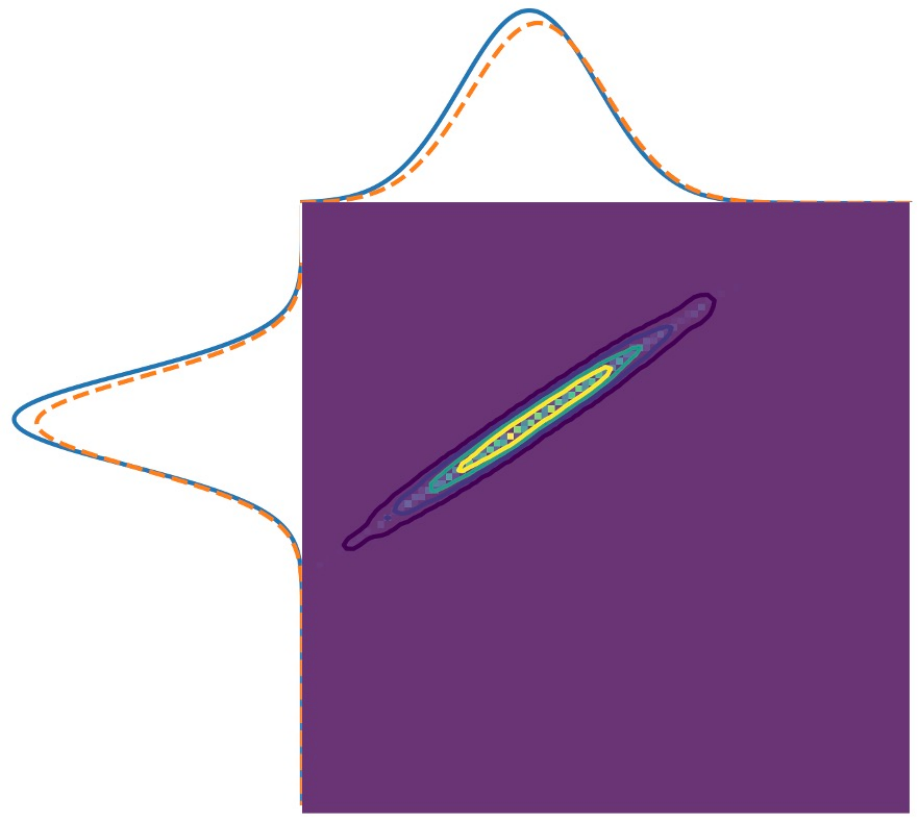}
\includegraphics[scale=0.23]{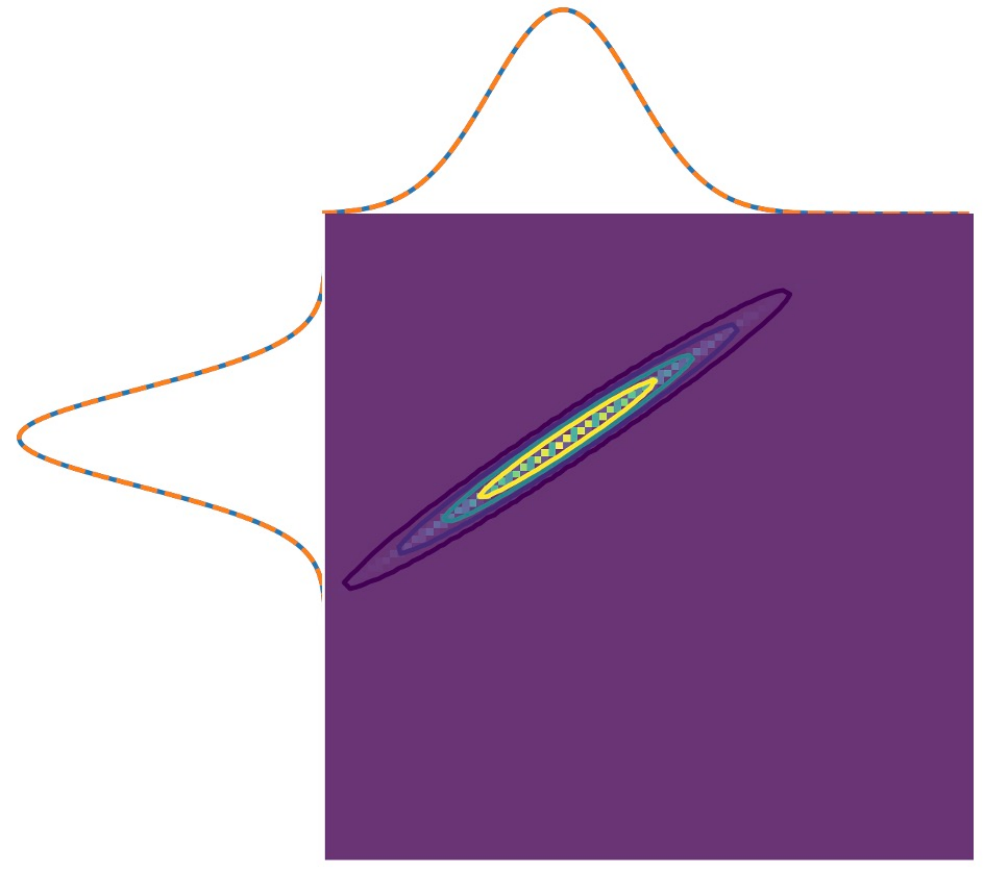}
\caption{Optimal transport plans for the UOT problem with $\gamma=0.2$ (left), \(\gamma =30\) (middle), and for standard OT (right).}
\label{fig:optimal-plan}
\end{figure*}

We now turn to the optimal transport plan in Case~2, where the two reference measures have the same total mass. In this case, the UOT problem is expected to recover the behavior of standard balanced OT as \(\gamma\) becomes large, because the KL penalties increasingly force the marginals of the plan to remain close to the prescribed reference measures. Figure~\ref{fig:optimal-plan} shows the optimal transport plan and its associated marginals for Case~2. The figure confirms the expected trend: as \(\gamma\) increases, the optimal plan becomes progressively closer to a standard OT coupling. Consistently, Table~\ref{tab:optimization_stats_fw} shows that the optimal transported mass \(c^*\) tends to \(1\) as \(\gamma\) increases. The behavior for small \(\gamma\) admits a complementary interpretation in terms of the transport cost. The first term in \eqref{eq:uot} favors transport over short distances and penalizes transport over longer distances; accordingly, the cost is smallest near the diagonal and increases away from it. When \(\gamma\) is small, the KL penalties on the marginals are relatively weak, so the optimizer prioritizes low-cost transport rather than close agreement with the reference measures. This is precisely what is observed in Fig.~\ref{fig:optimal-plan}: For \(\gamma=0.2\), the optimal transport plan is concentrated near the diagonal.

\subsection{UDC}\label{subsec:simulation-udc}
Next, we consider a UDC problem. The reference measures are \(\alpha = 1.0\,\mathcal{N}(-4.0,0.9^2),\) and \(\beta  = 0.4\,\mathcal{N}(4.0,0.6^2).\) The system parameters are set to \(A=B=1\). Applying Algorithm~\ref{alg:UDC} to this problem, Figure \ref{fig:UDC:state:meas:evol} (left and right) display the resulting optimal state-measure evolution for \(\gamma=3\) and \(\gamma=10\), respectively. In these figures, the solid red curve denotes the reference measure \(\alpha\), the solid blue curve denotes the reference measure \(\beta\), the dashed pink curve denotes the optimal initial measure \(\pi_1^*\), and the dashed light-blue curve denotes the optimal terminal measure \(\pi_T^*\). The thin solid grey curves represent the intermediate optimal state measures \(\{\pi_k^*\}\) for \(k=2,\dots,T-1\). As in the UOT example, the parameter \(\gamma\) controls the strength of the KL penalties that enforce proximity to the reference measures. Accordingly, as \(\gamma\) increases, the optimal initial and terminal measures \(\pi_1^*\) and \(\pi_T^*\) move closer to \(\alpha\) and \(\beta\), respectively. At the same time, in both cases, the mean of \(\pi_1^*\) lies to the right of the mean of \(\alpha\), while the mean of \(\pi_T^*\) lies to the left of the mean of \(\beta\). This reflects the trade-off between endpoint matching and control effort: rather than matching the reference measures exactly, the optimizer shifts the initial and terminal distributions toward one another so as to reduce the overall control cost.
\begin{figure}[t]
    \centering
    \includegraphics[scale=0.185]{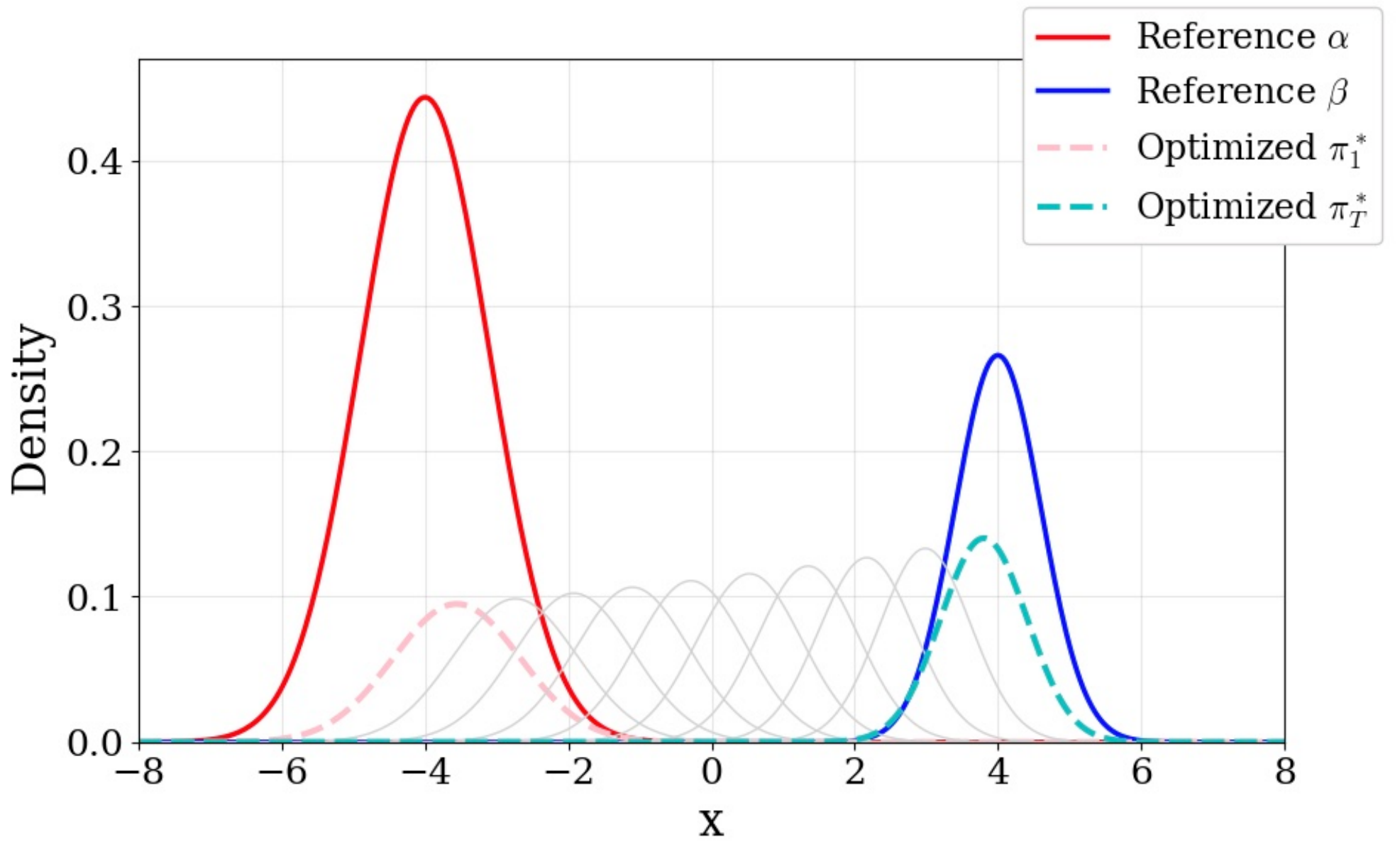}
    \includegraphics[scale=0.185]{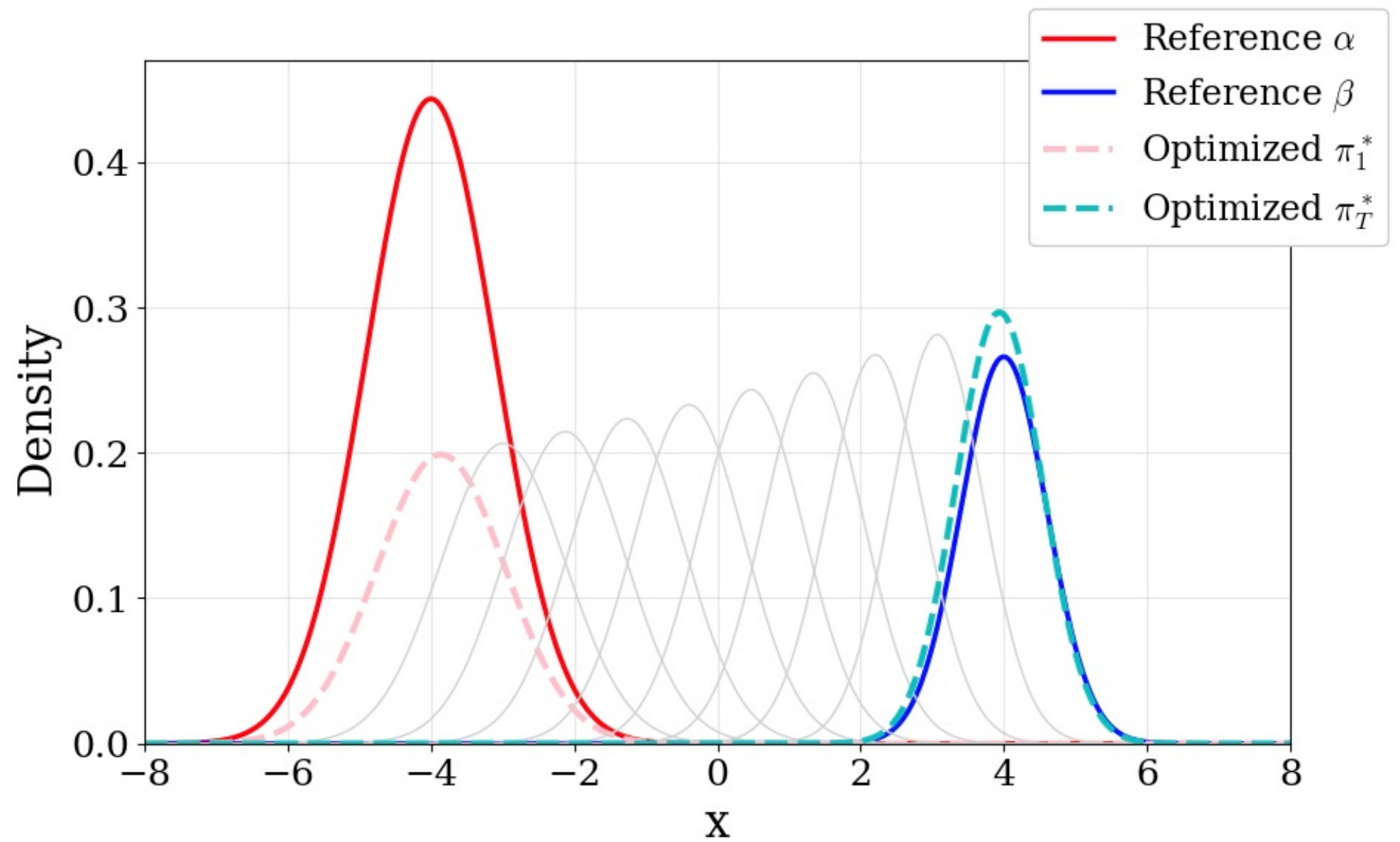}
    \caption{The reference measures \(\alpha\) and \(\beta\) with the optimal state measure \(\pi_k\) for UDC problem when \(\gamma=3.0\) (left) and \(\gamma = 10\) (right), computed by Algorithm~\ref{alg:UDC}. As \(\gamma\) increases, the optimized initial and terminal measures move closer to the prescribed reference measures.}
    \label{fig:UDC:state:meas:evol}
\end{figure}

\subsection{EUOT}\label{subsec:simulation-euot}
We present some numerical results for the entropic unbalanced optimal transport problem. We use the balanced reference measures $(\alpha_2,\beta_2)$ introduced in Case 2. For fixed $\gamma$, we solve the finite-dimensional Gaussian reformulation of Problem \ref{prob:euot} in \eqref{eq:finite-dim-euot} for several values of the entropic regularization parameter $\sigma$. The resulting optimal transport plans are shown in the top two subfigures in Fig.~\ref{fig:Euot_and_MaxEnt}. The parameter $\sigma$ controls the strength of the KL regularization term $D_{\mathrm{KL}}(\pi\|\alpha\otimes\beta)$. Hence, larger values of $\sigma$ encourage the optimal plan to stay closer to the independent product measure $\alpha\otimes\beta$. Since this term penalizes highly concentrated couplings, increasing $\sigma$ is expected to produce a transport plan with larger spread. This behavior is observed in Fig.~\ref{fig:Euot_and_MaxEnt}: for small $\sigma$, the optimal plan remains relatively concentrated, whereas for larger $\sigma$, the mass of the optimal plan is more broadly distributed. This result also highlights a qualitative difference between UOT and EUOT. In the UOT formulation, the optimal coupling between Gaussian marginals is induced by an affine transport map, and hence the resulting plan is concentrated along an affine relation in the small-penalty regime. In contrast, the EUOT formulation contains the additional entropy regularization term, which explicitly favors diffused transport plans. Therefore, even when the marginal reference measures are the same as those used in the UOT experiment, the resulting optimal coupling becomes smoother as $\sigma$ increases.

\begin{figure}[]
\centering
\includegraphics[scale=0.3]{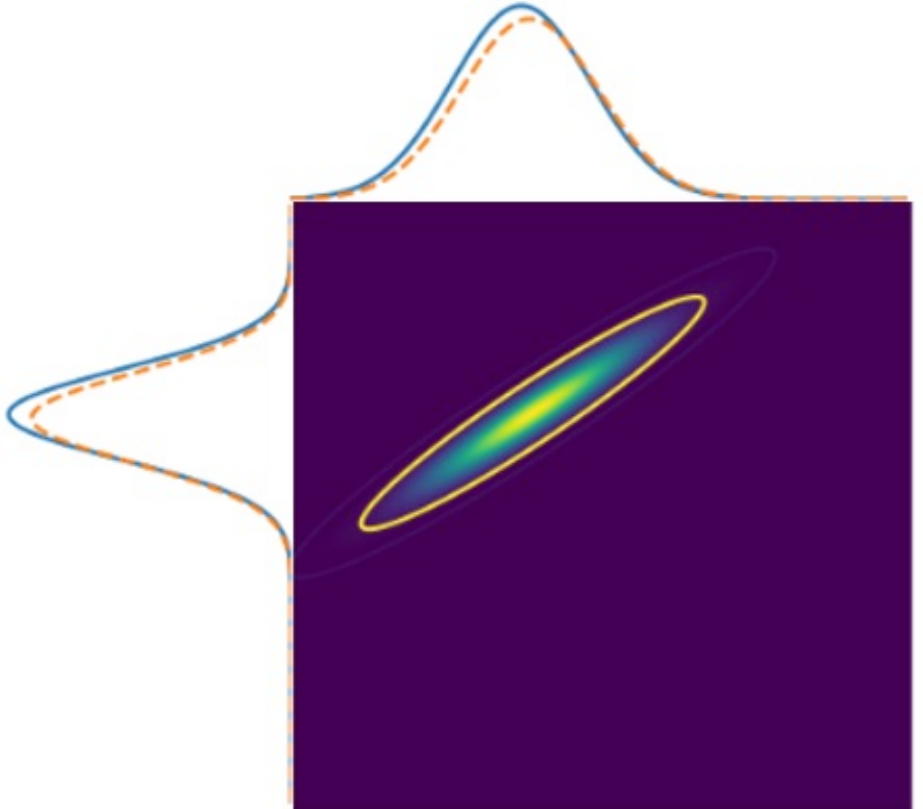}
\includegraphics[scale=0.29]{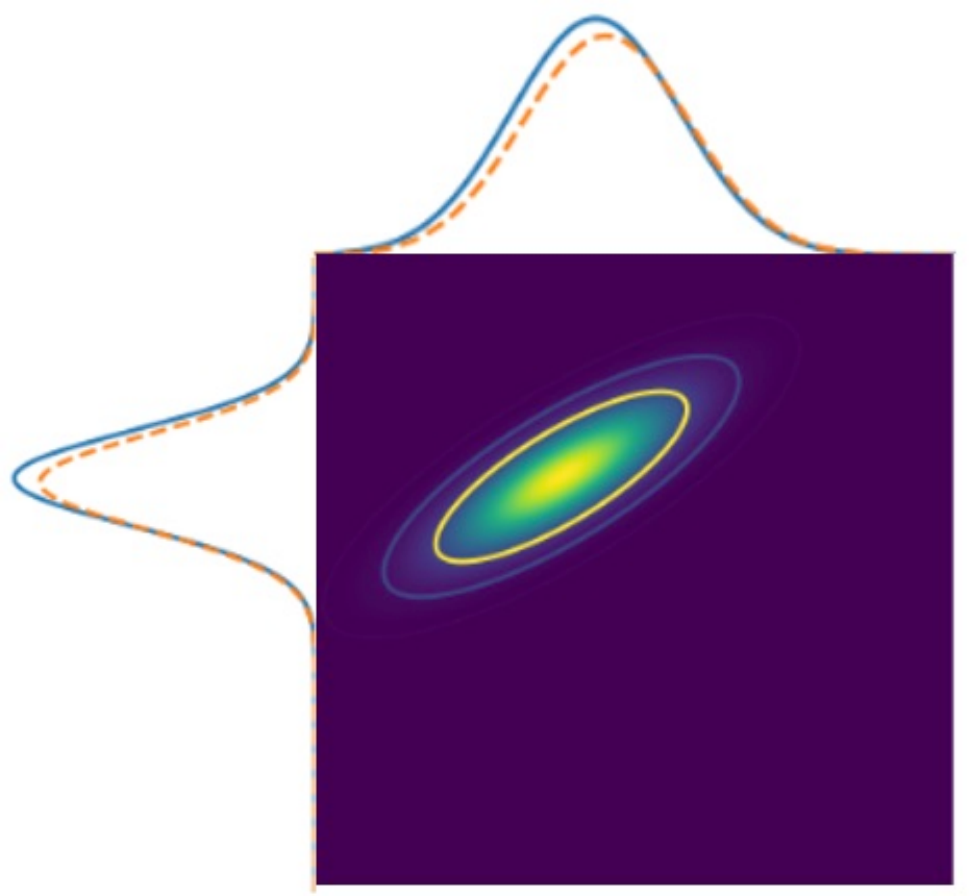}
\vspace{5mm}
\includegraphics[scale=0.3]{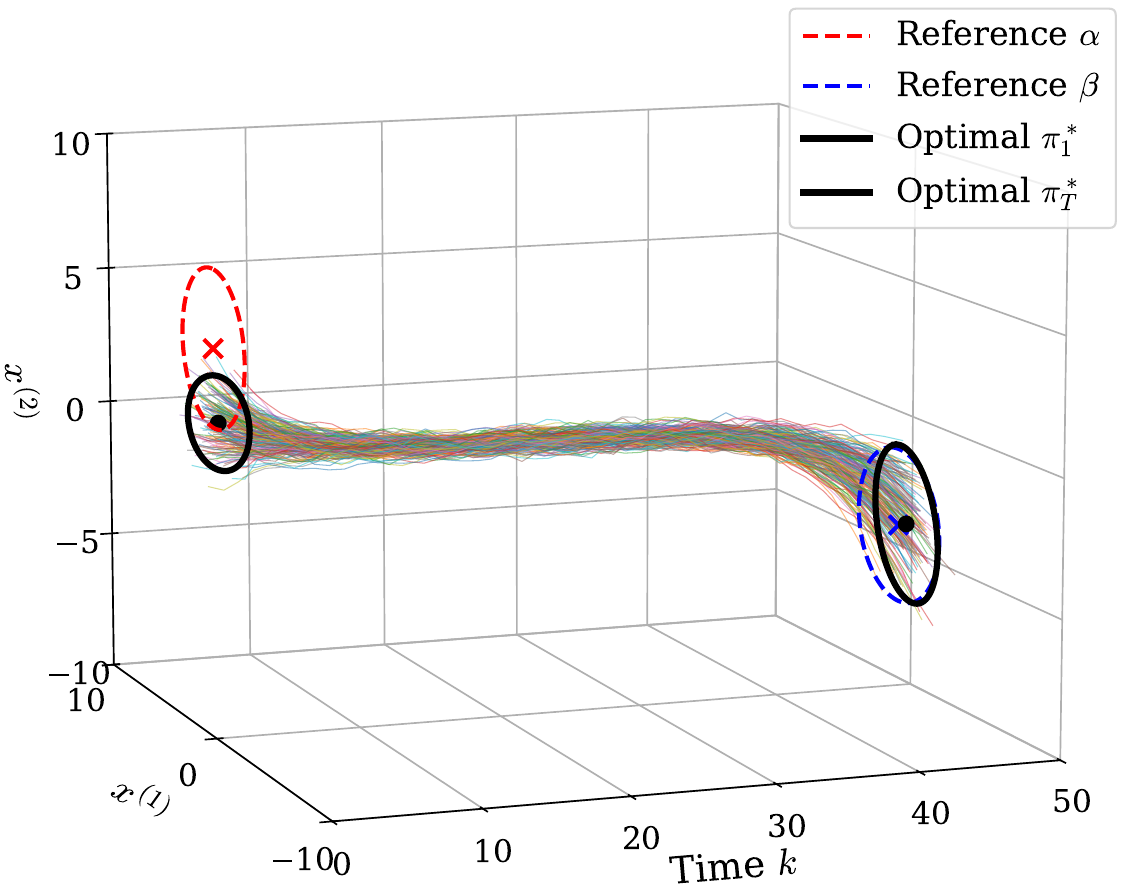}
\includegraphics[scale=0.3]{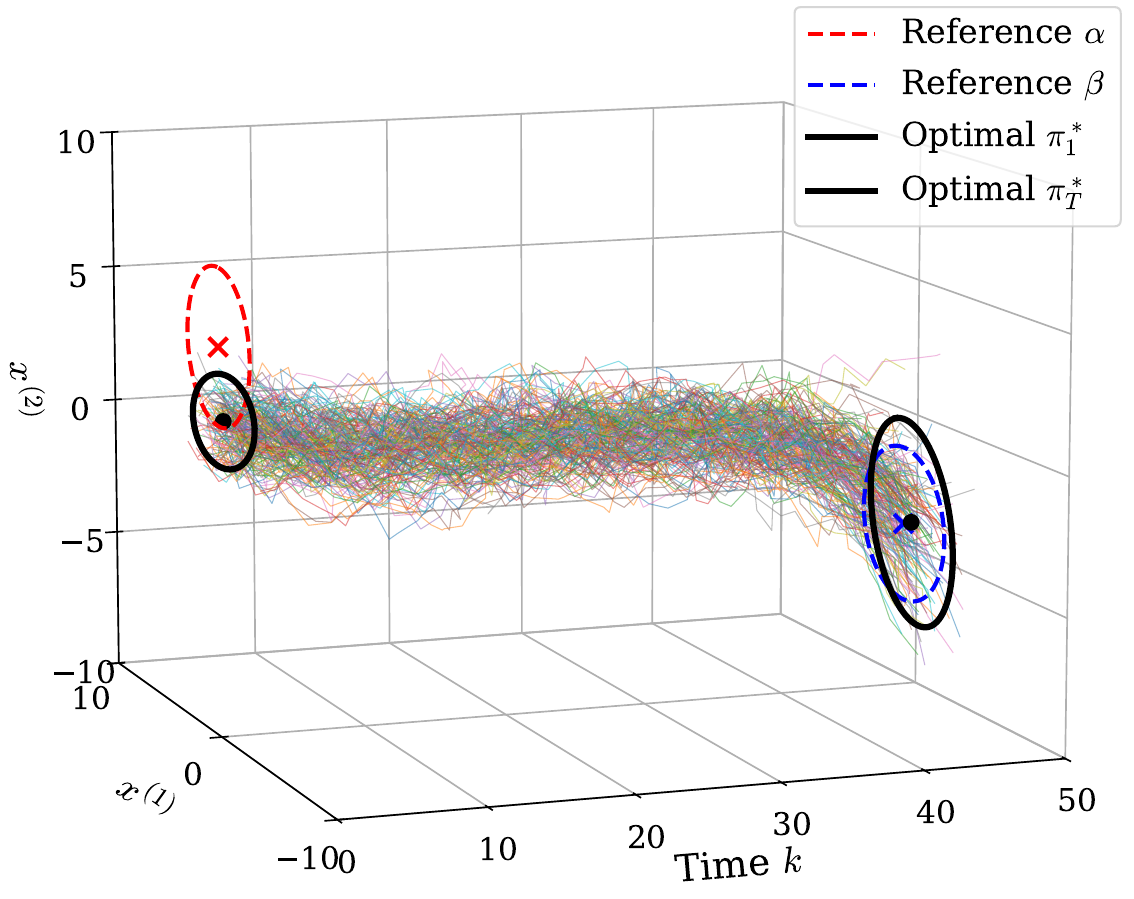}
\caption{(Top) The optimal transport plans for the entropic UOT problem. For \(\gamma=30\), the left-hand sub figure depicts the result when \(\sigma=0.1\), and the right-hand subfigure when \(\sigma=0.5\). (Right) 200 samples of the optimal state process for \eqref{eq:sim:maxent:system-mat} and \eqref{eq:sim:maxent:reference-measure}. For fixed \(\gamma=1.0\), the left-hand subfigure depicts the result when \(\eps=0.02\), and the right-hand figure when \(\eps=0.4\).}
\label{fig:Euot_and_MaxEnt}
\end{figure}

\subsection{MaxEnt UDC}\label{subsec:simulation-maxent}
Finally, we consider the maximum-entropy unbalanced density control problem in the two-dimensional case. Our state and actuation matrices are given by
\begin{align}\label{eq:sim:maxent:system-mat}
    A \Let \begin{bmatrix}
        0.9 & 0.1 \\ 0.05 & 1.2
    \end{bmatrix},\,\,B \Let I_2
\end{align}
The reference densities \(\alpha=\mathcal{N}(m_\alpha,\Sigma_\alpha)\) and \(\beta=\mathcal{N}(m_\beta,\Sigma_\beta)\) are determined by
\begin{align}\label{eq:sim:maxent:reference-measure}
   \hspace{-2mm} m_\alpha = \begin{pmatrix}
        0 \\ 4.0
    \end{pmatrix},\,
    m_\beta = \begin{pmatrix}
        0 \\ -4.0
    \end{pmatrix}, 
    \Sigma_\alpha = \Sigma_\beta=
    \begin{bmatrix}
        2.0 & 0 \\ 0 & 2.0
    \end{bmatrix}
\end{align}
and we set \(T=50\). We solve the relaxed finite-dimensional formulation \eqref{eq:maxent-relaxation} for different values of the entropy weight \(\eps>0\) with fixed \(\gamma\). The corresponding optimal state-measure evolutions are displayed in the the bottom subfigures in Fig.~\ref{fig:Euot_and_MaxEnt}.
Compared with the standard UDC result in \S\ref{subsec:simulation-udc}, the solution to Problem~\ref{prob:maxent_udc} contains an additional mechanism that encourages randomized control policies. Indeed, the objective of Problem \ref{prob:maxent_udc} includes the term $-\eps H(U_t(\cdot|x_t))$, and therefore a larger $\eps$ rewards higher-entropy control inputs. As a consequence, increasing $\eps$ leads to a larger control covariance and produces a more diffusive evolution of the state measure. This trend can be seen in Fig.~\ref{fig:Euot_and_MaxEnt}: when $\eps$ is small, the optimal control has a higher variances, and the intermediate state measures become more spread out.


\section{Conclusion}
\label{sec:conclusion}
In this paper, we developed a convex optimization-based global solution method for the UOT problem and formulated the UDC problem as a control-theoretic dynamical extension of UOT. We also studied entropy-regularized variants of these problems when the reference measures are Gaussian. A common feature of all these formulations is that, although they are originally posed as infinite-dimensional variational problems, their optimal solutions can be characterized within the Gaussian class. This enables an exact reduction to finite-dimensional optimization problems. While the resulting finite-dimensional problems are not necessarily convex in their original form, we showed that they can be solved globally by separating the mass optimization and applying lossless convex relaxations.

Several directions remain for future work. First, it would be interesting to apply the proposed framework to real-world tasks, such as single-cell RNA-seq data analysis \cite{morimoto2025}, where unbalanced transport is expected to be useful in the presence of observation corruption and outliers. Another important direction is to establish a connection with the unbalanced Schr\"odinger Bridge problem \cite{ref:Chen:UOT:2}. In the balanced setting, there is an elegant equivalence between SB problems and stochastic optimal control problems. In particular, in the discrete-time case, the MaxEnt optimal control problem is equivalent to an SB problem with a suitably chosen reference process when the initial and target distributions are Gaussian \cite{Ito2023}. Developing an unbalanced counterpart of this equivalence is an interesting topic for future research.

\bibliographystyle{amsalpha}
\bibliography{refs}

\end{document}